\documentclass[aap,preprint]{imsart}
\usepackage{amsmath,amssymb}
\usepackage[T1]{fontenc}      
\usepackage{geometry}         
\usepackage[francais,english]{babel}  
\usepackage{paralist}
\usepackage{mathrsfs}

\newtheorem{Theorem}{Theorem}[section]
\newtheorem{Definition}{Definition}[section]
\newtheorem{Proposition}{Proposition}[section]
\newtheorem{Assumption}{Assumption}[section]
\newtheorem{Lemma}{Lemma}[section]

\newtheorem{Remark}{Remark}[section]
\newtheorem{Example}{Example}[section]

\numberwithin{equation}{section}

\def \F{I\!\!F}
\def \H{I\!\!H}

\def \N{I\!\!N}
\def \R{I\!\!R}

\def\Fc{{\cal F}}

\def\L{{\cal L}}

\def\T{{\cal T}}

\def\Dzw1#1{\frac{\partial^2 #1}{\partial z \partial w_1}}

\def\Dzb1#1{\frac{\partial^2 #1}{\partial z \partial b_1}}

\newcommand{\dproof}{\noindent {Proof.} \quad}
\newcommand{\fproof}{\hfill $\square$ \bigskip}

\def\RB{\mathbb{R}}

\def\BC{\mathcal{B}}

\def\PC{\mathcal{P}}

\def\R{{\bf R}}

\def\1B{\text{1\!\!I}}

\def\tN{\tilde{N}}

\def\RB{\mathbb{R}}

\def\PC{\mathcal{P}}

\def\R{{\bf R}}

\def\1B{\text{1\!\!I}}

\def\tN{\tilde{N}}

\def\RB{\mathbb{R}}

\newcommand{\Q}{\mathbb{Q}}
\newcommand{\cf}{\mathcal{F}}
\newcommand{\ce}{\mathcal{E}}

\newcommand{\ct}{\mathcal{T}}
\newcommand{\stopt}{\mathcal{T}_{t,T}}
\newcommand{\stops}{\mathcal{T}_{S,T}}
\newcommand{\stopo}{\mathcal{T}_{0,T}}
\DeclareMathOperator*{\esssup}{ess\,sup}
\DeclareMathOperator{\e}{e}

\newcommand{\vertiii}[1]{{\left\vert\kern-0.25ex\left\vert\kern-0.25ex\left\vert #1 
    \right\vert\kern-0.25ex\right\vert\kern-0.25ex\right\vert}}

\newcommand{\vvertiii}[1]{{\vert\kern-0.25ex\vert\kern-0.25ex\vert #1 
    \vert\kern-0.25ex\vert\kern-0.25ex\vert}}

\newcommand{\triple}{\vert\kern-0.25ex\vert\kern-0.25ex\vert}

\newcommand{\BOX}{\ensuremath\Box}
\newenvironment{proof}{{\vskip\baselineskip\noindent\textbf{Proof:}}}%
{\hspace*{.1pt}\hspace*{\fill}\BOX\vskip\baselineskip}

\newcommand{\limsupn}{\limsup_{n\rightarrow\infty}}

\newcommand{\Ref}{{\mathcal{R}}ef}

\DeclareMathAlphabet{\mathpzc}{OT1}{pzc}{m}{it}

\begin{document}
\begin{frontmatter}


\title{Optimal stopping with $\lowercase \mathpzc{f}$-expectations: the irregular case} 
\runtitle{Optimal stopping with $\lowercase{f}$-expectations:  the irregular case}

\begin{aug}

\author{\fnms{Miryana} \snm{Grigorova}\thanksref{m1}
\ead[label=e1]{	
miryana\_grigorova@yahoo.fr}},
\author{\fnms{Peter} \snm{Imkeller}\thanksref{m2}\ead[label=e2]{imkeller@math.hu-berlin.de}},
\author{\fnms{Youssef} \snm{Ouknine}\thanksref{m3},
\ead[label=e3]{ouknine@ucam.ac.ma}
\ead[label=u1,url]{}}
\and
\author{\fnms{Marie-Claire} \snm{Quenez}\thanksref{m4}
\ead[label=e4]{quenez@lpsm.paris}}
\runauthor{M. Grigorova et al.}
\affiliation{ Bielefeld University \thanksmark{m1},  Humboldt University-Berlin \thanksmark{m2},\\ Université Cadi Ayyad \thanksmark{m3}, and Université Paris-Diderot \thanksmark{m4} }

\address{Centre for Mathematical Economics\\ Bielefeld University\\ Universitätsstr. 25\\ 33615 Bielefeld, Germany\\
\printead{e1}\\
}
\address{Institut für Mathematik\\
Humboldt Universität zu Berlin\\
Unter den Linden 6\\
10099 Berlin, Germany\\
\printead{e2}\\
}
\address{Département de Mathématiques\\
 Faculté des Sciences Semlalia\\
 Université Cadi Ayyad\\
 B.P. 2390\\
 Marrakech, Morocco\\
\printead{e3}\\}
\address{Laboratoire de Probabilités\\
 et Modèles Aléatoires\\
 Université Paris-Diderot\\
 Boîte courrier 7012\\
 75251 Paris Cedex 05, France\\
\printead{e4}\\
}

\end{aug}


\begin{abstract}
$\quad$ We consider the optimal stopping problem with non-linear $f$-expectation (induced by a BSDE) without making any regularity assumptions on the  payoff
 process $\xi$  and in the case of a general filtration. We show that the value family can be aggregated by an optional process $Y$. We characterize the process $Y$ as  the $\mathcal{E}^f$-Snell envelope of $\xi$. We also establish an infinitesimal characterization of the value process $Y$ in terms of a Reflected BSDE with $\xi$ as the obstacle.  
 To do this, we first establish some useful properties of irregular RBSDEs, in particular an existence and uniqueness result and a comparison theorem. 
\end{abstract}


\end{frontmatter}
\section{Introduction}\label{sec1}

The \textit{classical} optimal stopping probem with \textit{linear expectations} has been largely studied. General results on the topic can be found in El Karoui (1981) (\cite{EK}) where no regularity assumptions on the reward process $\xi$ are made. \\
In this paper, we are interested in a generalization of the classical optimal stopping problem where the linear expectation is replaced by a possibly non-linear functional, the so-called $f$-expectation ($f$-evaluation), induced by a BSDE with Lipschitz driver $f$. For a stopping time $S$ such that $0\leq S\leq T$ a.s. (where $T>0$ is a fixed terminal horizon), we define 
\begin{equation} \label{eq_intro}
V(S): =   {\rm ess} \sup_{\tau \in \T_{S,T}}{\cal E}^f_{S, \tau}(\xi_{\tau}),
\end{equation}
where  $\T_{S,T}$
 denotes the set of stopping times valued a.s. in $[S,T]$ and ${\cal E}^f_{S, \tau}(\cdot)$ denotes  the conditional $f$-expectation/evaluation at time $S$ when the   terminal time is $\tau$. 
 
 The above non-linear problem has been introduced in \cite{EQ96} in the case of a Brownian filtration and a continuous financial position/pay-off process $\xi$ and applied to the (non-linear) pricing of American options. It has then attracted considerable interest, in particular, due to its links with dynamic risk measurement (cf.,  e.g., \cite{Bayraktar-2}).  
 In the case of  a financial position/payoff process $\xi$, only supposed to be right-continuous,  this non-linear optimal stopping problem  has been studied
 in \cite{QuenSul2} (the case of Brownian-Poisson filtration), and in \cite{Bayraktar} 
where the non-linear expectation is supposed to be convex.
  To the best of our knowledge, \cite{MG} is the first paper addressing the stopping problem  \eqref{eq_intro} in the case of a non-right-continuous    process $\xi$  (with a Brownian-Poisson filtration); in \cite{MG} the  assumption of right-continuity of $\xi$ from the previous literature  is replaced by the weaker  assumption of right- uppersemicontinuity  (r.u.s.c.).\\
   In the present  paper, we study problem \eqref{eq_intro} in the case of a general filtration and without making any regularity assumptions on $\xi$, which allows for more flexibility in the modelling (compared to the cases of more regular payoffs and/or of particular filtrations).
   
  
The usual approach to address the classical optimal stopping problem (i.e., the case $f\equiv 0$ in  \eqref{eq_intro}) is a \emph{a direct approach}, based on  a direct study of the value family $(V(S))_{S\in\stopo}$. An important step in this approach is the aggregation of the value family by an optional process. The approach used in the literature to address the non-linear case (where $f$ is not necessarily equal to $0$) is \emph{an RBSDE-approach},  based on the study of a related Reflected BSDE  and on  linking directly the  solution of the Reflected BSDE with the value family $(V(S), S \in \T_{0,T})$ (and thus avoiding, in particular,  more technical aggregation questions). This approach (cf., e.g.,   \cite {MG}, \cite{QuenSul2}) requires at least the uppersemicontinuity of the reward process $\xi$  which we do not have here (cf. also Remark \ref{difficulty}).

Neither of the two approaches is applicable in the general framework of the present paper and we adopt \emph{a new approach which combines} some aspects of both the approaches.   
Our combined approach is the  following:  First, with the help of some results from the general theory of processes, we show that the value family $(V(S), S \in \T_{0,T})$ can be aggregated by a unique  right-uppersemicontinuous optional process $(V_t)_{t\in[0,T]}$. We characterize  the value process $(V_t)_{t\in[0,T]}$ as the  ${\cal E}^f$-{\em Snell envelope} of $\xi$, that is, the smallest strong ${\cal E}^f$-supermartingale greater than or equal to $\xi$.  Then,  we  turn to  establishing an infinitesimal characterization of the value process $(V_t)_{t\in[0,T]}$ in terms of a Reflected BSDE where the pay-off process $\xi$  from \eqref{eq_intro} plays the role of a lower obstacle. We emphasize that this \emph{RBSDE-part} of our approach is far from  mimicking  the one from the r.u.s.c. case;  we have to rely on very different  arguments here  due to the complete irregularity of the process $\xi$. 
 
Let us recall that Reflected BSDEs have been introduced  by El Karoui et al. in  the seminal paper \cite{ElKaroui97} in the case of a Brownian filtration and a continuous obstacle, and then generalized to the case of a  right-continuous obstacle and/or a larger stochastic basis than the Brownian one  in \cite{Ham}, \cite{CM}, \cite{HO1}, \cite{Essaky}, \cite{HO2}, \cite{QuenSul2}. In \cite{MG}, we have formulated a notion of Reflected BSDE in the case where the obstacle is only right-uppersemicontinuous (but possibly not right-continuous) and the filtration is the Brownian-Poisson filtration have shown existence and uniqueness of the solution.  In the present paper, we show that the existence and uniqueness result from \cite{MG} still holds in the  case  of a completely irregular obstacle and  a general filtration. In the recent preprint \cite{nouveau}, existence and uniqueness of the solution (in the Brownian framework) is shown by using a different approach, namely a penalization method.\\  
      We also establish a comparison result for RBSDEs with irregular obstacles {and general filtration}.
Due to the  complete irregularity of the obstacles and the presence of jumps, we are led to using  an approach which differs from those existing in the literature on comparison of RBSDEs (cf. also Remark \ref{Rmk_diff_0}); in particular, we first prove a generalization of Gal'chouk-Lenglart's formula (cf. \cite{Galchouk} and \cite{Lenglart})   to the case of convex functions, which we then  astutely apply in our framework  in order  to establish the comparison theorem. 
We also show an ${\cal E}^f$-Mertens decomposition for strong  ${\cal E}^f$-supermartingales, which generalizes to our framework the ones provided in the literature (cf. \cite{MG} or \cite{Bouchard}).
This result, together with our comparison theorem, 
  helps in the study  of the  non-linear  operator $\Ref^f$ which maps a given (completely irregular) obstacle to the solution of the RBSDE with driver $f$.  By using the properties of the operator $\Ref^f$,   we show that  $\Ref^f[\xi]$, that is, the (first component of the) solution to the Reflected BSDE with irregular obstacle $\xi$ and  driver $f$, is equal to the  ${\cal E}^f$-{\em Snell envelope} of $\xi$, from which we derive that it coincides with the value process $(V_t)_{t\in[0,T]}$ of problem \eqref{eq_intro}.  \\
 Finally, we give a financial application to the problem of pricing of American options with irregular pay-off in an imperfect market model. In particular,  we show that the superhedging price of the American  option with irregular pay-off $\xi$  is characterized as the solution of an associated RBSDE (where $\xi$ is the lower obstacle).  
Some examples of digital American options are
  given as 
particular cases.

The rest of the paper is organized as follows: In Section \ref{sec2} we give some preliminary definitions and some  notation. In Section  \ref{sec3} we revisit the classical optimal stopping problem  with irregular pay-off process $\xi$ and a general filtration.  We first give some general results such as aggregation, Mertens decomposition of the value process, Skorokhod conditions satisfied by the associated non decreasing processes; then,  we characterize the value process of the classical problem in terms of the solution of a Reflected BSDE  associated with a general filtration, with completely irregular obstacle and with a driver $f$ \emph{which does not depend on the solution}.  In Section \ref{subsect_RBSDEa}, we prove existence and uniqueness of the solution for \emph{general Lipschitz driver} $f$, an irregular obstacle $\xi$ and a general filtration.
In Section \ref{sec5}, we present the formulation or our \textit{non-linear} optimal stopping problem \eqref{eq_intro}. In Section \ref{rusc}, we provide some results on the particular case where the payoff $\xi$ is \emph{ right-uppersemicontinuous (r.u.s.c)}.\,, from which we derive an ${\cal E}^f$-Mertens decomposition 
of ${\cal E}^f$-strong supermartingales in the (general) framework of a general filtration (cf. Section \ref{Mertens2}). 
We then turn to the study of the case where $\xi$ is \textit{completely irregular}. Section \ref{irregularcase} is 
 devoted to the direct part of our approach to this problem
; in particular, we present the aggregation result and  the Snell characterization.  
 Section \ref{sec4} is devoted to establishing some properties of Reflected BSDEs with completely irregular obstacles, which will be used to establish an infinitesimal characterization of the value process of our problem \eqref{eq_intro} in the completely irregular case; more precisely,
we first provide a comparison theorem  (Subsection \ref{subsect_RBSDEb}); then, using this result together with the ${\cal E}^f$-Mertens decomposition,
we  establish  useful properties of the non-linear operator $\Ref^f$ (Subsection \ref{subsec_Ref}).
  In Section \ref{sec_infinitesimal}, using the results shown in the previous sections, we derive 
the infinitesimal characterization of the value of the non-linear optimal stopping problem \eqref{eq_intro} with a completely irregular payoff $\xi$ in terms of the solution of  our general RBSDE from Section \ref{subsect_RBSDEa}. 
In  Section \ref{sec_app} we give a financial application to the pricing of American options with irregular pay-off in an imperfect market model with jumps;  we also give a useful corollary of the infinitesimal characterization, namely, \emph{a priori estimates with universal constants} for RBSDEs with irregular obstacles and a general filtration. 



\section{Preliminaries} \label{sec2}

Let $T>0$ be a fixed positive real number.  Let  $E= \R^n\setminus\{0\}, \mathscr{E}= {\cal B}({\bf R}^n\setminus\{0\})$, which we    
equip with a $\sigma$-finite positive measure $\nu$.
Let $(\Omega,  \Fc, P)$ be a probability space equipped with a right-continuous complete filtration $\F = \{\Fc_t \colon t\in[0,T]\}$. Let $W$ be
a one-dimensional $\F$-Brownian motion $W$, and let $N(dt,de)$ an  $\F$-Poisson random measure  with compensator $dt\otimes\nu(de)$, supposed to be independent from $W$.
We denote by $\tilde N(dt,de)$ the compensated process, i.e. $\tilde N(dt,de):= N(dt,de)-dt\otimes\nu(de).$
We denote by ${\cal P}$ (resp. $\mathcal{O}$)   the predictable (resp. optional) $\sigma$-algebra
on $ \Omega\times [0,T]$.  The notation $L^2({\cal F}_T)$  stands for the space of random variables which are  $\Fc 
_T$-measurable and square-integrable. For $t\in [0,T],$ we denote by $\stopt$ the set of stopping times $\tau$ such that $P(t \leq\tau\leq T)=1.$ More generally, for a given stopping time $S\in \stopo$, we denote by $\ct_{S,T}$ the set of stopping times $ \tau$ such that $P(S \leq\tau\leq T)=1.$   \\





 We use also the following notation:\\ 
$\bullet$
$L^2_\nu$ is the set of $(\mathscr{E}, \mathcal{B}(\R))$-measurable functions $\ell:  E \rightarrow \R$ such that  $\|\ell\|_\nu^2:= \int_{ E}  |\ell(e) |^2 \nu(de) <  \infty.$
For $\ell\in\L^2_\nu$, $\mathpzc{k}\in\L^2_\nu$, we define $\langle \ell, \mathpzc{k}\rangle_\nu:=\int_E \ell(e)\mathpzc{k} (e) \nu(de)$.\\
$\bullet$    $\H^{2}$ is the set of 
$\R$-valued predictable processes $\phi$ with
 $\| \phi\|^2_{\H^{2}} := E \left[\int_0 ^T |\phi_t| ^2 dt\right] < \infty.$\\ 
$\bullet$ $\H_{\nu}^{2}$ is  the set of $\R$-valued processes $l: (\omega,t,e)\in(\Omega\times[0,T] \times  E)\mapsto l_t(\omega, e)$ which are {\em predictable}, that is $(\PC \otimes {\mathscr{E}},\BC(\R))$-measurable,
 and such that $\| l \|^2_{\H_{\nu}^{2}} :=E\left[ \int_0 ^T \|l_t\|_{\nu}^2 \,dt   \right]< \infty.$\\
$\bullet$ As in \cite{MG}, we  denote by  $ {\cal S}^{2}$ the vector space of $\R$-valued optional 
(not necessarily cadlag)
 processes $\phi$ such that
$\vertiii{\phi}^2_{{\cal S}^{2}} := E[\esssup_{\tau\in\T_0} |\phi_\tau |^2] <  \infty.$ By Proposition 2.1 in \cite{MG}, the mapping $\vertiii{\cdot}_{{\cal S}^{2}}$ is a norm on $ {\cal S}^{2}$,  and ${\cal S}^{2}$ endowed with this norm is a Banach space.  \\
$\bullet$ Let ${\cal M}^{2}$  be the set of square integrable  martingales $M= (M_t)_{t\in[0,T]}$ with $M_0=0$. This is a Hilbert space equipped with the scalar product 
$( M, M' ) _{{\cal M}^2}:= E[M_T M'_T]\,(= E[\, \langle M, M' \rangle_T]= E( \, [M, M' ]_T))$, for 
$M,M'$ $\in$ ${\cal M}^{2}$ (cf., e.g., \cite{Protter} IV.3). For each $M \in {\cal M}^{2}$, we set  $\| M\|^2_{{\cal M}^{2}}:= E( M_T^2)$. \\
$\bullet$ Let ${\cal M}^{2,\bot}$ be the subspace of martingales $h \in {\cal M}^{2}$
satisfying  $\langle h, W \rangle_\cdot =0$, and such that,
for all predictable processes $l$ $\in$ $\H_{\nu}^{2}$, 
\begin{equation} \label{ortho}
\langle h, \int_0^\cdot \int_E l_s(e) \tilde N(ds de) \rangle_t =0,\,\,\, 0 \leq t \leq T \quad {\rm a.s.}
\end{equation}
 
\begin{Remark} \label{2i} Note that condition \eqref{ortho} is equivalent to the fact that the square bracket process $[ \,h\,,  \int_0^\cdot \int_E l_s(e) \tilde N(ds de)\,]_t$ is a martingale (cf. the Appendix for additional comments on condition \eqref{ortho}).

Recall  also that the condition $\langle h, W \rangle_\cdot =0$ is equivalent to the orthogonality of $h$ (in the sense of the scalar product $( \cdot, \cdot ) _{{\cal M}^2}$) with respect to 
all stochastic integrals of the form 
$\int_0^\cdot z_s dW_s$, where 
$z \in \H^{2}$ (cf.\, e.g.\,, \cite{Protter} IV. 3 Lemma 2). 
Similarly, the condition \eqref{ortho} is equivalent to the orthogonality of $h$ with respect to 
all stochastic integrals of the form 
$\int_0^\cdot \int_E l_s(e) \tilde N(ds de)$, where 
$l$ $\in$ $\H_{\nu}^{2}$ (cf., e.g. , Lemma \ref{orthogonality} in the Appendix).
\end{Remark}
 We recall the following orthogonal decomposition property of martingales in 
${\cal M}^2$
(cf. Lemma III.4.24 in \cite{JS}). 
\begin{Lemma}\label{theoreme representation}
For each $M \in {\cal M}^2$, there exists a unique triplet 
$(Z,l,h) \in \H^{2} \times \H_{\nu}^{2} \times {\cal M}^{2, \bot}$ such that
\begin{equation}
\label{equation representation}
M_t  = \int_0^t Z_sdW_s+\int_0^t \int_{E} l_t(e) \tilde N(dt,de) +h_t\,, \quad\forall\,t\in[0,T] \quad a.s. 
\end{equation}
\end{Lemma}

%
\begin{Definition}[Driver, Lipschitz driver]\label{defd}
A function $f$ is said to be a {\em driver} if 
\begin{itemize}
\item  
$f: \Omega  \times [0,T]  \times \R^2 \times L^2_\nu \rightarrow \R $\\
$(\omega, t,y, z, \mathpzc{k}) \mapsto  f(\omega, t,y, z, \mathpzc{k})  $
  is $ \mathcal{P} \otimes {\cal B}(\R^2)  \otimes {\cal B}(L^2_\nu) 
- $ measurable,  
\item $E[\int_0^T f(t,0,0,0)^2dt] < + \infty$.
\end{itemize} 
A driver $f$ is called a {\em Lipschitz driver} if moreover there exists a constant $ K \geq 0$ such that $dP \otimes dt$-a.e.\,, 
for each $(y_1, z_1, \mathpzc{k}_1)\in \R^2 \times L^2_\nu$, $(y_2, z_2, \mathpzc{k}_2)\in \R^2 \times L^2_\nu$, 
$$|f(\omega, t, y_1, z_1, \mathpzc{k}_1) - f(\omega, t, y_2, z_2, \mathpzc{k}_2)| \leq 
K (|y_1 - y_2| + |z_1 - z_2| +   \|\mathpzc{k}_1 - \mathpzc{k}_2 \|_\nu).$$
\end{Definition}

\begin{Definition}[BSDE, conditional $f$-expectation]\label{BSDE}

We have (cf., e.g., Remark \ref{EDSRclassique} in the Appendix) that if   $f$ is a Lipschitz driver and if $\xi$ is  in $L^2({\cal F}_T)$, then there exists a unique solution $(X,\pi, l,h)\in{\cal S}^2 \times \H^2 \times \H^2_\nu\times {\cal M}^{2, \bot}$ to the following BSDE:\\
$-dX_t= f(t, X_t,  \pi_t, l_t)dt-  \pi_t dW_t-\int_E  l_t(e) \tilde N(dt,de)-dh_t; \quad X_T= \xi.$\\
For $t\in[0,T]$, the (non-linear) operator $\mathcal{E}^f_{t,T}(\cdot): L^2(\cf_T)\rightarrow L^2(\cf_t)$ which maps a given terminal condition
$\xi\in L^2(\cf_T)$ to the position $X_t$ (at time $t$) of the first component  of the solution of the above BSDE 
is called  \emph{conditional $f$-expectation at time $t$}.  As usual, this notion  can be extended to
the case where the (deterministic) terminal time 
$T$
is replaced by a (more general) stopping time $\tau\in\stopo$,  $t$ is replaced by a stopping time $S$ such that $S\leq \tau$ a.s. and  the domain $L^2(\cf_T)$ of the operator    is replaced  by $L^2(\cf_\tau)$.

\end{Definition}
We now pass to the notion of Reflected BSDE. Let $T>0$ be a fixed terminal time. Let $f$ be  a  driver. 
Let $\xi= (\xi_t)_{t\in[0,T]}$ be a   process in ${\cal S}^2$. 

We define the  process $(\overline {\xi}_{t})_{t\in ]0,T]}$ by $\overline{\xi}_t:= \limsup_{s \uparrow t, s< t} \xi_s,$ for all $t\in]0,T].$ We recall that $\overline \xi$ is a predictable process (cf. \cite[Thm. 90, page 225]{DM1}). The process $\underline \xi$ is left upper-semicontinuous and is called the left upper-semicontinuous envelope of $\xi$.

\begin{Definition}[Reflected BSDE] \label{def_solution_RBSDE}
A process $(Y,Z,k,h,A,C)$ is said to be a solution to the reflected BSDE with parameters $(f,\xi)$, where $f$ is a driver and $\xi$ is a  process in ${\cal S}^2$, if, 
\footnote{As usual, equation \eqref{RBSDE} means that a.s.\,, for all $t\in[0,T]$, we have:\\
$$Y_t=Y_T+\int_{t}^T f(s,Y_s,  Z_s,k_s)ds-\int_{t}^T  Z_s dW_s-\int_{t}^T \int_E k_s(e) \tilde N(ds,de)  -  h_T+ h_t+A_T-A_t+C_{T-} -C_{t-}.$$}
\begin{align}\label{RBSDE}
&(Y,Z,k,h,A,C)\in {\cal S}^2 \times \H^2 \times \H^2_\nu \times {\cal M}^{2, \bot} \times {\cal S}^2\times {\cal S}^2,   
\nonumber \\
&-d  Y_t  =  f(t,Y_{t}, Z_{t}, k_t)dt + dA_t + d C_{t-}-  Z_t dW_t-\int_E k_t(e)\tilde{N}(dt,de) - dh_t, \,0\leq t\leq T,
\\
& Y_T=\xi_T \,\text{ a.s.,} \,\text{ and}\,\, \,Y_t \geq \xi_t   \text{ for all } t\in[0,T],  \text{ a.s.,} \nonumber\\
& A \text{ is a nondecreasing right-continuous predictable process 
with } A_0= 0 \text{ and such that } \nonumber\\
& \int_0^T {\bf 1}_{\{Y_{t-} > \overline{\xi}_{t}\}} dA^c_t = 0 \text{ a.s. and } \; (Y_{\tau-}-\overline{\xi}_{\tau})(A^d_{\tau}-A^d_{\tau-})=0 \text{ a.s. for all predictable }\tau\in\stopo, \label{rRBSDE_A}\\   
& C \text{ is a nondecreasing right-continuous adapted purely discontinuous process with } C_{0-}= 0  \nonumber\\
& \text{ and such that }
(Y_{\tau}-\xi_{\tau})(C_{\tau}-C_{\tau-})=0 \text{ a.s. for all }\tau\in\stopo. \label{rRBSDE_C}
\end{align}
\end{Definition}
Here  $A^c$ denotes the continuous part of the process $A$ and $A^d$  its discontinuous part.\\
Equations \eqref{rRBSDE_A} and \eqref{rRBSDE_C} are referred to as \emph{minimality conditions} or \emph{Skorokhod conditions}.

For real-valued random variables $X$ and $X_n$, $n \in \N$, the notation   "$X_n\uparrow X$" stands  for "the sequence $(X_n)$ is nondecreasing and converges to $X$ a.s.".  \\
For a ladlag process $\phi$, we denote by $\phi_{t+}$ and $\phi_{t-}$ the right-hand and left-hand limit of $\phi$ at  $t$. We denote by $\Delta_+ \phi_t:=\phi_{t_+}-\phi_t$ the size of the right jump of $\phi$ at $t$, and by   $\Delta \phi_t:=\phi_t-\phi_{t-}$ the size of the left jump of $\phi$ at $t$.

\begin{Remark}\label{rRmk_left_uppersemicontinuous envelope}
In the particular case where $\xi$ has left limits, we can replace the process $(\overline{\xi}_t)$ by the process of left limits $(\xi_{t-})$ in the Skorokhod condition   \eqref{rRBSDE_A}. 
\end{Remark}

\begin{Remark}\label{Rmk_the_jumps_of_C}
If $(Y,Z,k,h,A,C)$ is a solution to the RBSDE defined above, by \eqref{RBSDE}, we have $\Delta C_t=Y_t-Y_{t+}$, which implies that $Y_t\geq Y_{t+}$, for all $t\in[0,T)$.  Hence, $Y$ is r.u.s.c. Moreover, from $C_\tau-C_{\tau-}=-(Y_{\tau+}-Y_{\tau})$, combined with the Skorokhod condition \eqref{rRBSDE_C}, we derive $(Y_\tau-\xi_\tau)(Y_{\tau+}-Y_{\tau})=0$, a.s. for all $\tau\in\stopo$.    This, together with $Y_\tau\geq \xi_\tau$ and $Y_\tau\geq Y_{\tau+}$ a.s., leads to $Y_\tau= Y_{\tau+}\vee \xi_\tau$ a.s. for all $\tau\in\stopo$.  
%
\begin{Definition} \label{defr} 
%
%
 
 Let $\tau \in \T_0$. An optional process $(\phi_t)$ is said to be {\em right upper-semicontinuous (resp. left upper-semicontinuous})  {\em along stopping times} if for all stopping time $\tau \in \T_0$  and for all non increasing (resp. non decreasing) sequence of stopping times $ (\tau_n)$ such that $\tau^n \downarrow \tau$ (resp. $\tau^n \uparrow \tau$) a.s.\,,
$\phi_{\tau} \geq \limsup_{n\to \infty} \phi_{\tau_n} \;\; \rm{a.s.}$.  


\end{Definition}

\end{Remark}

\section{The classical optimal stopping problem}\label{sec3}

%
In this section, we revisit the classical (linear) optimal stopping problem with irregular pay-off process and a general filtration. 
\subsection{The classical linear optimal stopping problem revisited}\label{subsection_linear}
 Let $(\xi_t)_{t\in[0,T]}$ be a 
 process belonging to ${\cal S}^2$, called the {\em reward process} or the {\em pay-off process}.
For each $S \in \stopo$,  we define the value  $ v  (S)$ at time $S$ by
\begin{eqnarray}\label{ddeux-2_prel}
 v  (S):= \esssup_{\tau \in \stops} E[ \xi_{\tau} \mid \Fc_S].
\end{eqnarray} 
\begin{Lemma} \label{pro_prel}


 \begin{description}
 \item[(i)]  There exists a ladlag optional process $( v_t)_{t\in[0,T]}$ which aggregates the family 
 $(v(S))_{S\in\stopo}$ (i.e. $ v_S= v (S)$ a.s. for all $S\in\stopo$).\\
Moreover, the process $( v_t )_{t\in[0,T]}$ is the smallest strong supermartingale
 greater than or equal to $( \xi_t)_{t\in[0,T]}$.
 \item[(ii)] We have 
$ v_S=\xi_S\vee  v_{S+}$ a.s. 
   for all $S\in  {\cal T}_{0,T}$.
 \item[(iii)]\footnote{Note that in the case of a not necessarily non-negative pay-off process $\xi$ this result holds up to a translation by the martingale $X_S:=E[\esssup_{\tau\in\stopo} \xi_\tau^-|\cf_S]$ (cf. e.g. Appendix A in \cite{Marc}). More precisely, the property holds for $\tilde{v}:=v+X$ and $\tilde\xi=\xi+X$.}  For each $S \in  {\cal T}_{0,T}$ and for each $\lambda \in ]0,1[$, 
  the process $( v_t )_{t\in[0,T]}$ is a  martingale on 
$[S, \tau_S^{\lambda}]$, where  $\tau^{\lambda}_S:= \inf\{t\geq S \,, \lambda { v}_t(\omega)\leq \xi_t \}$.
 \end{description} 
 \end{Lemma}
 
 \dproof
 These results are due to  classical results of optimal stopping theory. For a sketch of the proof of 
 the first two  assertions, the reader is referred to the proof of Proposition A.5 in the Appendix of \cite{MG} (which still holds for a general process $\xi \in {\cal S}^2$).
 The last assertion corresponds to a result of optimal stopping theory (cf. \cite{Maingueneau}, \cite{EK} or Lemma 2.7 in \cite{Kob}). Its proof 
is based on a penalization method (used in convex analysis), introduced by Maingueneau (1978) (cf. the proof of Theorem 2 in \cite{Maingueneau}), which does not require any regularity assumption on the reward process $\xi$. 
\fproof

\begin{Remark}\label{sautY_prel}
It follows from $(ii)$ in the above lemma that 
$\Delta_+  v_{S}= {\bf 1}_{\{ v_S= \xi_S\}}\Delta_+  v_{S}$ a.s.
\end{Remark}

\begin{Remark}\label{Rk_Maingueneau}
Let us note for further reference that Maingueneau's penalization approach for showing  the martingale property on 
$[S,\tau_S^{\lambda}]$ (property (iii) in the above lemma)   relies heavily on the convexity of the problem. 
\end{Remark}

\begin{Lemma}\label{propertiesvalue_prel}  \begin{description}
 \item[(i)]   
The value  process ${ V}$ of Lemma \ref{pro_prel} 
 belongs to ${\cal S}^{2}$ and admits the following (Mertens) decomposition:
\begin{equation}\label{eq_Mertens_decomposition_Y_prel}
{ v}_t=v_0+M_t-A_t-C_{t-}, \text{ for all $t\in[0,T]$ a.s.}, 
\end{equation}
where $M \in {\cal M}^2$, $A$ is a nondecreasing right-continuous predictable process 
such that $ A_0= 0$, $E(A_T^2)<\infty$, and $C$ is a nondecreasing right-continuous adapted purely discontinuous process such that  $C_{0-}= 0$, $E(C_T^2)<\infty$. 
 \item[(ii)]  
For each $\tau \in\stopo$, we have
 $\Delta C_{\tau}= {\bf 1}_{\{{ v}_\tau= \xi_\tau\}}\Delta C_{\tau}$ a.s.\,
 \item[(iii)]  For each predictable  $\tau \in\stopo$, we have 
$\Delta A_{\tau}= {\bf 1}_{\{{ v}_{\tau-}= \, \overline{\xi}_{\tau}\}}\Delta A_{\tau}$ a.s.\,
\end{description}
\end{Lemma}

\dproof    By Lemma \ref{pro_prel} (i), the process $({ v}_t )_{t\in[0,T]}$ is a strong supermartingale.
Moreover,  by using martingale inequalities, it can be shown 
that
$E[\esssup_{S\in\stopo}|{ V}_S|^2]\leq  c\vertiii{\xi}^2_{{\cal S}^{2}}.$
Hence,  the process $({ v}_t )_{t\in[0,T]}$ is in $\mathcal{S}^2$ (a fortiori, of class   (D)). Applying
Mertens decomposition for strong supermartingales of class (D) (cf., e.g.,  
 \cite[Appendix 1, Thm.20, equalities 
(20.2)]{DM2}) 
  gives the decomposition \eqref{eq_Mertens_decomposition_Y_prel}, where $M$ is a cadlag uniformly integrable martingale, $A$ is a nondecreasing right-continuous predictable process 
such that $ A_0= 0$, $E(A_T)<\infty$, and $C$ is a nondecreasing right-continuous adapted purely discontinuous process such that  $C_{0-}= 0$, $E(C_T)<\infty$.  Based on some results of Dellacherie-Meyer \cite{DM2} (cf., e.g., Theorem A.2 and Corollary A.1 in \cite{MG}), we derive that $A \in\mathcal{S}^2$ and $C \in\mathcal{S}^2$, 
which gives the assertion  (i).
  
 Let $\tau \in\stopo$. 
By Remark \ref{sautY_prel} together with
 Mertens decomposition \eqref{eq_Mertens_decomposition_Y_prel}, we get $\Delta C_{\tau}= - \Delta_+ { v}_{\tau}$ a.s. It follows that 
$\Delta C_{\tau}= {\bf 1}_{\{{ v}_\tau= \xi_\tau\}}\Delta C_{\tau}$ a.s.\,, 
which corresponds to  (ii).

%

Assertion  (iii) (concerning the jumps of $A$) is due to 
 El Karoui \footnote{Note that the proof in El Karoui \cite{EK} is given for nonnegative pay-off $\xi$.  To pass from this to the more general case where $\xi$ might take also negative values, we apply  the result by El Karoui \cite{EK} with $\tilde \xi:= \xi+X$ (which is non-negative) and $\tilde v:=v+X$, where the process $X=(X_t)$ is defined by  $X_t:=E[\esssup_{\tau\in\stopo} \xi_\tau^-|\cf_t]$. We then notice that the
 Mertens process $(A,C)$ from the Mertens decomposition of $v$ is the same as the Mertens process ($\tilde A$, $\tilde C$) from the Mertens decomposition of $\tilde v$ (indeed, only the martingale parts of the two decompositions differ by $X$).   Moreover, we see that the set ${\{{ v}_{\tau-}= \, \overline{\xi}_{\tau}\}}$ is the same as the set where $v$ is replaced by $\tilde v$ and $\xi$ is replaced by $\tilde \xi$ (this is due to the fact that $X$ is a martingale and thus has left limits; so $\overline X_t=X_{t-}$).} (\cite[Proposition 2.34]{EK})   
Its proof is based on the equality $A_S= A_{\tau_S^{\lambda}}$ a.s.\,, for each 
$S \in  {\cal T}_{0,T}$ and for each $\lambda \in ]0,1[$ (which follows from Lemma \ref{pro_prel} (iii) together with 
 Mertens decomposition \eqref{eq_Mertens_decomposition_Y_prel}).

\fproof

The following minimality property for the continuous part $A^c$ is well-known from the literature in the "more regular" cases (cf., e.g.,   \cite{ERRATUM}    for the right-uppersemicontinuous case). In the case of completely irregular $\xi$, this minimality property  was  not explicitly available.  Only recently, it was proved by  \cite{nouveau} (cf. Proposition 3.7) in the Brownian framework. Here, we generalize the result of \cite{nouveau} to the case of a general filtration by using different analytic arguments. 

\begin{Lemma}\label{propertiesvalue_Ac_prel}
The continuous part $A^c$ of $A$ satisfies the equality $\int_0^T {\bf 1}_{\{{ v}_{t-} > \overline{\xi}_{t}\}} dA^c_t = 0$ a.s.\, 
\end{Lemma}
\dproof
 As for the discontinuous part of $A$, the proof is based on Lemma \ref{pro_prel} (iii)
, and  also on some analytic arguments similar 
to those used in the proof of Theorem D13 in Karatzas and Shreve (1998) (\cite{KS2}). 

We have to show that 
$\int_0^T ({ v}_{t-} - \overline{\xi}_{t}) dA^c_t = 0$  a.s. \\
Lemma \ref{pro_prel} (iii) yields that
 for each 
$S \in  {\cal T}_{0,T}$ and for each $\lambda \in ]0,1[$, we have
$
A_S= A_{\tau_S^{\lambda}} \quad {\rm a.s.}
$
Without loss of generality, we can assume that for each $\omega$, 
the map $t\mapsto A^c_t(\omega)$ is continuous, that the map $t\mapsto { v}_t(\omega)$ is  left-limited, and that, for all $\lambda \in ]0,1[ \cap  \Q$ and 
$t \in  [0, T[ \cap \Q$, we have $A_t(\omega) = A_{\tau_t^{\lambda}} (\omega)$.

Let us denote by $\mathcal{J}(\omega)$ the set on which the nondecreasing function $t \mapsto A^c_t(\omega)$ is ``flat'':
\begin{equation*}\label{Jo}
\mathcal{J}(\omega) : = \{t \in ]0,T[ \, , \,\, \exists \delta>0\,\,\, \mbox{with} \,\,\, A^c_{t -\delta} (\omega)= A^c_{t +\delta}(\omega)  \}
\end{equation*}
The set $\mathcal{J}(\omega)$ is clearly open and hence can be written as a countable union of disjoint intervals: 
${\displaystyle \mathcal{J}(\omega)  = \cup_i ]\alpha_i (\omega), \beta_i(\omega)[}$.
We consider 
\begin{equation}\label{definitionhatJ}
\mathcal{\hat J}(\omega) :  = \cup_i ]\alpha_i(\omega), \beta_i(\omega)] = \{t \in ]0,T] \, , \,\, \exists \delta>0\,\,\, \mbox
{with} \,\,\, A^c_{t -\delta} (\omega)= A^c_{t}(\omega)  \}.
\end{equation}
We have $\int_0^T {\bf 1} _{\mathcal{\hat J}(\omega) } dA^c_t (\omega)= 
\sum_{i} (A^c_ {\beta_i(\omega) }(\omega)- A^c_ {\alpha_i(\omega) }(\omega)) =0.$ Hence, 
the nondecreasing function $t \mapsto A^c_t(\omega)$ is  ``flat'' 
on $\mathcal{\hat J}(\omega)$.
We now introduce 
\[\mathcal{K}(\omega):=\{t \in  ]0,T] \,\,
{\rm s.t.} \,\, { v}_{t-} (\omega) > \overline\xi_{t}(\omega)\} \]
We next show that for almost every $\omega$,
$\mathcal{K}(\omega) \subset \mathcal{\hat J}(\omega),$
which clearly provides the desired result.
Let $t \in \mathcal{K}(\omega)$. Let us prove that $t \in  \mathcal{\hat J}(\omega)$. By
\eqref{definitionhatJ}, we thus have to show that there exists $\delta >0$ such that 
$A^c_{t -\delta} (\omega)= A^c_t (\omega)$.
Since $t \in \mathcal{K}(\omega)$, we have ${ v}_{t-} (\omega) > \overline\xi_{t}(\omega)$. Hence, 
there exists $\delta >0$ and $\lambda \in ]0,1[ \cap  \Q$ such that $t-\delta \in  [0, T[ \cap \Q$ and
for each $ r \in [t-\delta, t[$, $\lambda { v}_r(\omega) > \xi_r(\omega) $.
By definition of $\tau_{t- \delta}^{\lambda} (\omega)$, it follows that
 $\tau_{t- \delta}^{\lambda} (\omega) \geq t$. 
Now, we have $A^c_{\tau_{t- \delta}^{\lambda}} (\omega)=A^c_{t -\delta} (\omega)$. Since 
the map $s \mapsto A^c_s (\omega)$ is nondecreasing, we get  $A^c_t (\omega)=A^c_{t -\delta} (\omega)$, which implies that $t \in  \mathcal{\hat J}(\omega)$.
We thus have $\mathcal{K}(\omega) \subset \mathcal{\hat J}(\omega)$, which completes the proof.
\fproof
\begin{Remark}\label{Rk_minimal}
We note  that the martingale property from assertion $(iii)$ of Lemma \ref{pro_prel} is crucial for the proof of the minimality  conditions for the process $A$  (namely, for the proofs of Lemma \ref{propertiesvalue_prel} assertion $(iii)$, and for Lemma \ref{propertiesvalue_Ac_prel}). 
\end{Remark}


\subsection{The classical linear optimal stopping problem with an additional instantaneous reward} 

In this subsection, we extend the previous results to the case where, besides the reward process  $\xi$,  there is an additional  running (or instantaneous) reward process $f\in\H^2$. More precisely, let $(\xi_t)_{t\in[0,T]}$ be a 
 process belonging to ${\cal S}^2$, called the {\em reward process} or the {\em pay-off process}.
Let $f=(f_t)_{t\in[0,T]}$ be a predictable process with $E[\int_0^T f_t^2dt] < + \infty$, called the {\em instantaneous reward process}. For each $S \in \stopo$,  we define the value  $ V  (S)$ at time $S$ by
\begin{eqnarray}\label{ddeux-2}
 V  (S):= \esssup_{\tau \in \stops} E[ \xi_{\tau} + \int_S^\tau f_u du \mid \Fc_S].
\end{eqnarray} 
This is equivalent to 

\begin{eqnarray}
 V  (S)+\int_0^S f_u du:= \esssup_{\tau \in \stops} E[ \xi_{\tau} + \int_0^\tau f_u du \mid \Fc_S].
\end{eqnarray}
Hence, the results of the previous subsection can be applied  with $\xi_\cdot$ replaced by $\xi_\cdot+\int_0^\cdot f_udu$  and $v(S)$ replaced by $V(S)+  \int_0^S f_udu$. Here is a brief summary. 
\begin{Lemma} \label{pro}


 \begin{description}
 \item[(i)]  There exists a ladlag optional process $( V_t)_{t\in[0,T]}$ which aggregates the family 
 $(V(S))_{S\in\stopo}$ (i.e. $ V_S= V (S)$ a.s. for all $S\in\stopo$).\\
Moreover, the process $( V_t + \int_0^t f_u du)_{t\in[0,T]}$ is the smallest strong supermartingale
 greater than or equal to $( \xi_t + \int_0^t f_u du)_{t\in[0,T]}$.
 \item[(ii)] We have 
$ V_S=\xi_S\vee  V_{S+}$ a.s. 
   for all $S\in  {\cal T}_{0,T}$.
 \end{description} 
 \end{Lemma}
 

\begin{Remark}\label{sautY}
It follows from $(ii)$ in the above lemma that 
$\Delta_+  V_{S}= {\bf 1}_{\{ V_S= \xi_S\}}\Delta_+  V_{S}$ a.s.
\end{Remark}

\begin{Lemma}\label{propertiesvalue}  \begin{description}
 \item[(i)]  
The value  process ${ V}$ of Lemma \ref{pro} 
 belongs to ${\cal S}^{2}$ and admits the following (Mertens) decomposition:
\begin{equation}\label{eq_Mertens_decomposition_Y}
{ V}_t=V_0-\int_0^t f_u du+M_t-A_t-C_{t-}, \text{ for all $t\in[0,T]$ a.s.}, 
\end{equation}
where $M \in {\cal M}^2$, $A$ is a nondecreasing right-continuous predictable process 
such that $ A_0= 0$, $E(A_T^2)<\infty$, and $C$ is a nondecreasing right-continuous adapted purely discontinuous process such that  $C_{0-}= 0$, $E(C_T^2)<\infty$. 
 \item[(ii)]  
For each $\tau \in\stopo$, we have
 $\Delta C_{\tau}= {\bf 1}_{\{{ V}_\tau= \xi_\tau\}}\Delta C_{\tau}$ a.s.\,
 \item[(iii)]  For each predictable  $\tau \in\stopo$, we have 
$\Delta A_{\tau}= {\bf 1}_{\{{ V}_{\tau-}= \, \overline{\xi}_{\tau}\}}\Delta A_{\tau}$ a.s.\,
\end{description}
\end{Lemma}

\begin{Lemma}\label{propertiesvalue_Ac}
The continuous part $A^c$ of $A$ satisfies the equality $\int_0^T {\bf 1}_{\{{ V}_{t-} > \overline{\xi}_{t}\}} dA^c_t = 0$ a.s.\, 
\end{Lemma}

\subsection{Characterization of the value function as the solution of an RBSDE}\label{subsect_stop2}
In this subsection, 
we show, using the above lemmas, that the value process $V$ of the classical  optimal stopping problem  \eqref{ddeux-2} solves the RBSDE from Definition \ref{def_solution_RBSDE} with parameters  the driver process $(f_t)$ and the obstacle $(\xi_t)$. We also 
prove the uniqueness of the solution of this RBSDE.
 To this aim, we first provide  \emph{a priori} estimates for RBSDEs in our general framework. 
\begin{Lemma}[A priori estimates]\label{Lemma_estimate}
Let $(Y^1, Z^1, k^1, h^1, A^1, C^1)$ (resp. $(Y^2, Z^2, k^2, h^2, A^2, C^2)$) $\in {\cal S}^2 \times \H^2 \times \H^2_\nu \times {\cal M}^{2, \bot} \times {\cal S}^2\times {\cal S}^2$  be  a  solution to the RBSDE  associated with  driver $f^1(\omega, t)$ (resp. $f^2(\omega, t)$) and with obstacle $\xi$. We  set $\tilde Y:=Y^1-Y^2$, $\tilde Z:=Z^1-Z^2$, $\tilde A:=A^1-A^2$, $\tilde C:=C^1-C^2$,  $\tilde k:=k^1-k^2$, $\tilde h:=h^1-h^2$, and $\tilde f(\omega, t):=f^1(\omega, t)-f^2(\omega, t)$. There exists $c>0$ such that for all $\varepsilon>0$,  for all $\beta\geq\frac 1 {\varepsilon^2}$ we have 
\begin{align}\label{eq_initial_Lemma_estimate}
\|\tilde Z\|^2_\beta\leq \varepsilon^2\|\tilde f\|^2_\beta, \,\, \|\tilde k\|^2_{\nu,\beta}\leq \varepsilon^2\|\tilde f\|^2_\beta  \text{  and  } \|\tilde h \|^2_{\beta, {\cal M}^2} \leq \varepsilon^2\|\tilde f\|^2_\beta.  
\\
 \vvertiii{\tilde Y}^2_\beta \leq 4\varepsilon^2(1+12c^2) \|\tilde f\|^2_\beta.\label{deuxeq}
\end{align}
\end{Lemma}
\dproof The proof is given in the Appendix.
\fproof

 Using these {\em a priori} estimates, the lemmas from the previous subsection, and the orthogonal martingale decomposition  (Lemma \ref{theoreme representation}), we derive the following "infinitesimal characterization" of the value process $V$.
%

\begin{Theorem}\label{rf}
Let $V$ be the value process of the optimal stopping problem \eqref{ddeux-2}. 
Let $A$ and $C$ be the non decreasing processes associated with the Mertens decomposition \eqref{eq_Mertens_decomposition_Y} of $V$.
There exists a unique triplet $(Z,k,h)\in  \H^2 \times \H^2_\nu\times {\cal M}^{2, \bot}$ such that the process $(V,Z,k,h,A,C)$ is a solution of 
the RBSDE from Definition \ref{def_solution_RBSDE} associated with the driver process $f(\omega,t, y,z,\mathpzc{k}) = f_t(\omega)$ and the obstacle $(\xi_t)$. Moreover, the solution of this RBSDE is unique.
\end{Theorem}

\dproof 
By Lemma \ref{pro} (ii), the value   process $V$ corresponding to the optimal stopping problem \eqref{ddeux-2} satisfies  ${ V}_T= { V} (T)= \xi_T\,$ a.s.  
and ${ V}_t\geq \xi_t,$ $0 \leq t \leq T$, a.s. 
By Lemma \ref{propertiesvalue}  (ii), the process $C$ of the Mertens decomposition of $ V$ 
\eqref{eq_Mertens_decomposition_Y}
 satisfies the minimality condition \eqref{rRBSDE_C}.
Moreover, by Lemma \ref{propertiesvalue} (iii) and Lemma \ref{propertiesvalue_Ac},
 the process $A$ satisfies the minimality condition 
\eqref{rRBSDE_A}. 
By Lemma \ref{theoreme representation}, 
 there exists a unique triplet
$(Z,k,h) \in \H^{2} \times \H_{\nu}^{2} \times {\cal M}^{2, \bot}$ such that
$dM_t = Z_t  dW_t + \int_{ E} k_t(e) \tilde{N}(dt,de)+dh_t.$ 
The process $(V,Z,k,h,A,C)$  is thus a solution of the RBSDE
\eqref{def_solution_RBSDE} associated with the driver process $(f_t)$ and the obstacle $\xi$. 

It remains to show the uniqueness  of the solution. 
 Using the \textit{a priori estimates} from  Lemma \ref{Lemma_estimate}, together with classical arguments  
 (cf. step 5 of the proof of Lemma 3.3 in \cite{MG}), we obtain the desired result.
\fproof

We are  interested in generalizing this result to the case of the optimal stopping problem \eqref{eq_intro} with \emph{non-linear} $f$-expectation  (associated with a non-linear driver $f(\omega, t,y,z,\mathpzc{k})$).
To this purpose, we first establish an existence and uniqueness result for the RBSDE from Definition \ref{def_solution_RBSDE} in the case of a general (non-linear) Lipschitz driver $f(\omega, t,y,z,\mathpzc{k})$. 

\section{Existence and uniqueness of the solution of the RBSDE with an irregular obstacle and a general filtration in the case of a general driver}\label{subsect_RBSDEa}
In Theorem \ref{rf}, we have shown that, in the case where the driver does not depend on $y,z,$ and $\mathpzc{k}$, the RBSDE from Definition \ref{def_solution_RBSDE} admits a unique solution. 
Using this result together with the above {\em a priori } estimates from Lemma \ref{Lemma_estimate},  we derive the following existence and uniqueness result in the case of a general Lipschitz driver $f(t,y,z,k)$. 
\begin{Theorem}[Existence and uniqueness]\label{rexiuni}
Let  $\xi$  be a 
 process in $\mathcal{S}^2$ and let $f$ be a  Lipschitz driver. 
The RBSDE with parameters $(f,\xi)$ from Definition \ref{def_solution_RBSDE} admits a unique solution $(Y,Z,k,h,A,C)\in \mathcal{S}^2  \times \H^2 \times \H^2_\nu \times \mathcal{M}^{2,\bot}\times \mathcal{S}^2\times \mathcal{S}^2.$
\end{Theorem}
\dproof 
For each $\beta>0$, we denote by $\mathcal{B}_\beta^2$ the Banach space $\mathcal{S}^2 \times \H^2\times \H^2_\nu$ which we  equip with the norm $\|(\cdot,\cdot,\cdot) \|_{\mathcal{B}_\beta^2}$ defined by 
$\| (Y, Z, k)\|_{\mathcal{B}_\beta^2}^2:=\vvertiii {Y}_{\beta}^2  +   \| Z\|_{\beta}^2+\|k \|_{\nu,\beta}^2, $ for  $(Y,Z,k)\in \mathcal{S}^2 \times \H^2\times \H^2_\nu.$ 
%
We define a mapping $\Phi$ from $\mathcal{B}_\beta^2$ into
itself as follows: for a given $(y, z, l) \in \mathcal{B}_\beta^2$, we set $\Phi (y, z, l):= (Y, Z, k)$, where $Y, Z, k$ 
are
the first three components of the solution $(Y, Z, k,h,A,C)$ to the RBSDE  associated with driver process $f(s):= 
f(s, y_s, z_s, l_s)$ and with obstacle $\xi$. 
The mapping $\Phi$ is well-defined by Theorem \ref{rf}. Using the a priori estimates from Lemma \ref{Lemma_estimate} and similar computations as those from the proof of Theorem 3.4 in \cite{MG}, we derive that $\Phi$ is a contraction for the norm $\| \cdot \|_{\mathcal{B}_\beta^2}$. By the fixed point theorem in the Banach space 
$\mathcal{B}_\beta^2$, the mapping $\Phi$ thus admits a unique fixed point, which corresponds to the unique solution of the RBSDE with parameters $(f,\xi)$.
\fproof
\begin{Remark}
In  \cite{nouveau}, the above  existence and uniqueness result  is shown in a Brownian framework by using a penalization method. Our approach  provides an alternative proof of this result. 
\end{Remark}

We now provide a useful property of the solution of an RBSDE, which will be used in the sequel.
\begin{Lemma}[$\mathcal{E}^f$-martingale property of $Y$]\label{tol}
Let $\xi$ be a 
 process in $\mathcal{S}^2$ and let $f$ be a Lipschitz driver. 
Let 
$(Y,Z,k,h,A,C)$ be the solution to the reflected BSDE with parameters $(f,\xi)$ as in Definition \ref{def_solution_RBSDE}. For each $S \in  {\cal T}_{0,T}$ and for each $\varepsilon>0$, we set
 \begin{equation}\label{definitiontau}
 \tau^{\varepsilon}_S:= \inf\{t\geq S \,,  { Y}_t\leq \xi_t + \varepsilon\}.
 \end{equation} 
The process $(Y_t)$ is 
 an ${\cal E}^{f}$-martingale on $[S, \tau^{\varepsilon}_S]$. 
\end{Lemma} 

\begin{proof} The proof in our case of a general filtration is identical to that of Lemma 4.1 (statement $(ii)$) in \cite{MG}  and is given here  for the convenience of the reader\footnote{We note that the proof of Lemma 4.1 (statement $(ii)$) in \cite{MG}  does not require the assumption of r.u.s.c. of $\xi$.}. By definition of $\tau^{\varepsilon}_S$, we have: for a.e. $\omega\in\Omega $, for all $t\in[S(\omega), \tau^{\varepsilon}_S(\omega)[$, $Y_t(\omega)> \xi_t(\omega)  + \varepsilon.$  
Hence, by the Skorokhod condition for $A$, we have that  for a.e. $\omega\in\Omega $, the function $t\mapsto A^c_t(\omega) $ is constant on $[S(\omega), 
\tau^{\varepsilon}_S(\omega)[$; by continuity of almost every trajectory of the process $A^c$, $A^c_\cdot(\omega)$ is constant on the closed interval $[S(\omega), 
\tau^{\varepsilon}_S(\omega)]$, for a.e. $\omega$. Furthermore, (again by the Skotokhod condition for $A$), for a.e. $\omega\in\Omega $, the function  $t\mapsto A^d_t(\omega)$ is constant on $[S(\omega), \tau^{\varepsilon}_S(\omega)[.$
Moreover, $ Y_{ (\tau^{\varepsilon}_S)^-  } \geq \xi_{ (\tau^{\varepsilon}_S)^-  }+ \varepsilon\,$ a.s.\,, which implies that $\Delta A^d _ {\tau^{\varepsilon}_S  } =0$ a.s.\, Finally, for a.e. $\omega\in\Omega $, for all $t\in[S(\omega), \tau^{\varepsilon}_S(\omega)[$, $\Delta C_t(\omega)=C_t(\omega)-C_{t-}(\omega)=0$; 
therefore, for a.e. $\omega\in\Omega $, for all $t\in[S(\omega), \tau^{\varepsilon}_S(\omega)[$, $\Delta_+ C_{t-}(\omega)=C_t(\omega)-C_{t-}(\omega)=0$, which implies that, for a.e. $\omega\in\Omega$, the function $t\mapsto C_{t-}(\omega)$ is constant on  $[S(\omega), \tau^{\varepsilon}_S(\omega)[.$ By left-continuity of almost every trajectory of the process  $(C_{t-})$, we get that for a.e. $\omega\in\Omega$, the function $t\mapsto C_{t-}(\omega)$ is constant on the closed interval $[S(\omega), \tau^{\varepsilon}_S(\omega)]$. Thus, for a.e. $\omega\in\Omega$,  the map $t \mapsto A_t (\omega) + C_{t-} (\omega)$ is constant on $[S(\omega), \tau^{\varepsilon}_S(\omega)]$. Hence, $Y$ is the solution on $[S, 
\tau^{\varepsilon}_S]$ of the BSDE associated with driver $f$, terminal time $\tau^{\varepsilon}_S$ and terminal condition $Y_{\tau^{\varepsilon}_S}.$ The result follows.
 \end{proof}
 
\begin{Remark}\label{Rmk_positif}
Note that in the case where $\xi$ is nonnegative, the above result holds true also on the stochastic interval $[S,\tau_S^\lambda]$, where $\lambda\in(0,1)$  and 
$\tau_S^\lambda:=\inf\{t\geq S: \lambda Y_t\leq \xi_t\}$. Note that in the case of non-negative obstacle, we have also $Y\geq 0$ (as $Y\geq \xi\geq 0$);  hence, $\lambda Y_T\leq Y_T=\xi_T$ a.s. and $ \tau_S^\lambda$ is finite a.s.  

\end{Remark}

\section{Optimal stopping with non-linear $f$-expectation: formulation of the problem}\label{sec5}

Let $(\xi_t)_{t\in[0,T]}$ be a  process in $\mathcal{S}^2$. Let $f$ be a Lipschitz driver.  
For each $S\in$  $\stopo$, we  define \emph{the value  at time} $S$ by 
\begin{equation} \label{vvv}
V(S): =   {\rm ess} \sup_{\tau \in \T_{S,T}}{\cal E}^f_{S, \tau}(\xi_{\tau}).
\end{equation}

 
We make the following assumption on the driver (cf., e.g.,  Theorem 4.2 in \cite{QuenSul}). 
 
 \begin{Assumption}\label{Royer}
Assume that  $dP \otimes dt$-a.e.\, for each $(y,z, \mathpzc{k}_1,\mathpzc{k}_2)$ $\in$ $ \RB^2 \times (L^2_{\nu})^2$,
$$f( t,y,z, \mathpzc{k}_1)- f(t,y,z, \mathpzc{k}_2) \geq \langle \theta_t^{y,z, \mathpzc{k}_1,\mathpzc{k}_2}  \,,\,\mathpzc{k}_1 - \mathpzc{k}_2 \rangle_\nu,$$ 
where 
$
\theta:  [0,T]  \times \Omega\times \RB^2 \times  (L^2_{\nu})^2  \rightarrow  L^2_{\nu}\,; \, (\omega, t, y,z, \mathpzc{k}_1,\mathpzc{k}_2)\mapsto 
\theta_t^{y,z, \mathpzc{k}_1,\mathpzc{k}_2}(\omega,\cdot)
$ is a
 ${\cal P } ·\otimes {\cal B}(\R^2) \otimes  {\cal B}( (L^2_{\nu})^2 )$-measurable mapping,
satisfying
$\|\theta_t^{y,z, \mathpzc{k}_1,\mathpzc{k}_2}(\cdot)\|_{\nu}  \leq C\,$  for all $(y,z, \mathpzc{k}_1,\mathpzc{k}_2)$ $\in$ $\RB^2 \times (L^2_{\nu})^2$, $ \,dP\otimes dt $-a.e.\,, where $C$ is a positive constant, and such that  
 $
\theta_t^{y,z, \mathpzc{k}_1,\mathpzc{k}_2} (e)\geq -1,
$
for all $(y,z, \mathpzc{k}_1,\mathpzc{k}_2)$ $\in$ $\RB^2 \times (L^2_{\nu})^2$, $\,dP\otimes dt \otimes d\nu(e)-{\rm a.e.}$
\end{Assumption} 

The above assumption is satisfied if, for example, $f$ is of class ${\cal C}^1$ with respect to $\mathpzc{k}$ such that $\nabla_\mathpzc{k} f$ is bounded (in $L^2_{\nu}$) and $\nabla_\mathpzc{k} f \geq -1$ (cf. 
Proposition A.2. in \cite{DQS2}).

We recall that under Assumption \ref{Royer} on the driver $f$, the functional ${\cal E}^f_{S, \tau}(\cdot)$ is nondecreasing (cf. \cite[Thm. 4.2]{QuenSul} and Remark \ref{EDSRclassique}).\\
As mentioned in the introduction, the above optimal stopping problem has been largely studied: in \cite{EQ96}, and in \cite{Bayraktar-2}, in the case of a continuous pay-off process $\xi$;   in \cite{QuenSul2} and \cite{Bayraktar}   in the case of a right-continuous pay-off; and recently in \cite{MG} in the case of a  right-uppersemicontinuous pay-off process $\xi$. In this section, we do not make any regularity assumptions on $\xi$ (cf. also Remark \ref{rRmk_left_uppersemicontinuous envelope}).\\
 If we interpret $\xi$ as a financial position process and $- {\cal E}^f(\cdot)$ as a dynamic risk measure (cf.,e.g.,  \cite{Pe04}, \cite{Gianin}), then (up to a minus sign) $V(S)$ can be seen  as the minimal risk at time $S$. 
As also mentioned in  the introduction, the absence of regularity  allows for more flexibility in the modelling. If, for instance, we consider a situation where the jump times of the Poisson random measure model times of default (which, being totally inaccessible, cannot be foreseen),  then, the complete lack of regularity allows to  take into account  an immediate non-smooth, positive or negative,  impact on $\xi$ after the default occurs.\\
If we interpret $\xi$ as a payoff  process, and ${\cal E}^f(\cdot)$ as a non linear pricing rule, then the optimal stopping problem \eqref{vvv} is related to the (non linear) pricing problem of the American option with payoff $\xi$. The absence of regularity allows us to deal with the case of American options with irregular payoffs, such as American digital options (cf. Section \ref{amerique} for details). On the other hand, the fact that  the filtration is not necessarily the natural filtration associated with $W$ and $N$   allows to incorporate some additional information in the  modelling (such as, for example, default risks or other economic factors).\\
We  begin by addressing the simpler case where the payoff is assumed to be \emph{right u.s.c.}\, This preliminary study of the right u.s.c. case will allow us to establish an $\mathcal{E}^f$-Mertens decomposition for strong $\mathcal{E}^f$-supermartingales with respect to a general filtration (extending the existing results from the literature; cf. \cite{Bouchard} and \cite{MG}). This  will be an important result for the treatment of the non-linear optimal stopping problem in the case of a  \emph{completely irregular} pay-off.

\section{Optimal stopping with non-linear $f$-expectation: the right u.s.c. case}\label{rusc}

Let $f$ be a Lipschitz driver satisfying Assumption \ref{Royer}. The following result relies crucially on an assumption of right-uppersemicontinuity of $\xi$. 
\begin{Lemma}\label{lemma_epsilon_optimality}
Let $\xi$ be a 
 process in $\mathcal{S}^2$, supposed to be \emph{right u.s.c.} Let 
$(Y,Z,k,h,A,C)$ be the solution to the reflected BSDE with parameters $(f,\xi)$ as in Definition \ref{def_solution_RBSDE}. Let $S \in  {\cal T}_{0,T}$ and let $\varepsilon>0$. Let $\tau^{\varepsilon}_S$ be the stopping time defined by 
 \eqref{definitiontau}, that is, 
 $\tau^{\varepsilon}_S:= \inf\{t\geq S \,,  { Y}_t\leq \xi_t+\varepsilon \}$. We have
\begin{equation}\label{lambdabis}
  { Y}_{ \tau_S^{\varepsilon} }\leq \, \xi_{  \tau_S^{\varepsilon}}+\varepsilon \quad   \rm{a.s.}
  \end{equation}
 \end{Lemma} 
\begin{proof}
%
 The proof of this result in our case of a general filtration is  identical  to that from  \cite[Lemma 4.1(i)]{MG} in the case of a Brownian-Poisson filtration.  We give  again the arguments here in order to emphasize the important role of the right-uppersemicontinuity assumption on $\xi$.   By way of contradiction, we suppose $P( Y_{\tau^{\varepsilon}_S}> \xi_{\tau^{\varepsilon}_S}+\varepsilon)>0$. By the Skorokhod condition for $C$, we have  $\Delta C_{\tau^{\varepsilon}_S}=C_{\tau^{\varepsilon}_S}- C_{(\tau^{\varepsilon}_S)-}=0$ on the set $\{ Y_{\tau^{\varepsilon}_S}> \xi_{\tau^{\varepsilon}_S}+\varepsilon\}$. On the other hand, due to Remark \ref{Rmk_the_jumps_of_C}, $\Delta C_{\tau^{\varepsilon}_S}=Y_{\tau^{\varepsilon}_S}-Y_{(\tau^{\varepsilon}_S)+}.$ Thus, $Y_{\tau^{\varepsilon}_S}=Y_{(\tau^{\varepsilon}_S)+}$ on the set $\{ Y_{\tau^{\varepsilon}_S}> \xi_{\tau^{\varepsilon}_S}+\varepsilon \}.$ Hence,
\begin{equation}\label{eq_contradiction}
\lambda Y_{(\tau^{\varepsilon}_S)+}> \xi_{\tau^{\varepsilon}_S}  \text{ on the set } \{Y_{\tau^{\varepsilon}_S}> \xi_{\tau^{\varepsilon}_S}+\varepsilon\}.
\end{equation}
 We will obtain a contradiction with this statement.
Let us fix $\omega\in\Omega$. By definition of $\tau^{\varepsilon}_S(\omega)$, there exists a non-increasing sequence $(t_n)=(t_n(\omega))\downarrow \tau^{\varepsilon}_S(\omega) $ such that $ Y_{t_n}(\omega)\leq \xi_{t_n}(\omega)+\varepsilon$, for all $n\in\N$. Hence,
$ \limsupn Y_{t_n}(\omega)\leq \limsupn \xi_{t_n}(\omega)+\varepsilon.$ As the process $\xi$ is \emph{right-uppersemicontinuous}\,, we have $\limsupn \xi_{t_n}(\omega)\leq \xi_{\tau^{\varepsilon}_S}(\omega)$. On the other hand, as $(t_n(\omega))\downarrow \tau^{\varepsilon}_S(\omega)$, we have $\limsupn Y_{t_n}(\omega)=Y_{(\tau^{\varepsilon}_S)+}(\omega).$ Thus, $ Y_{(\tau^{\varepsilon}_S)+}(\omega)\leq \xi_{\tau^{\varepsilon}_S}(\omega)+\varepsilon,$ which is in contradiction with \eqref{eq_contradiction}. We conclude that $ Y_{\tau^{\varepsilon}_S}\leq  \xi_{\tau^{\varepsilon}_S}+\varepsilon \text{ a.s.}$
\end{proof}
%
%
 With the help of the previous lemma together with Lemma \ref{tol}, we derive the following result.  
\begin{Theorem}[Characterization theorem in the  r.u.s.c.  case]\label{caracterisation}
 Let $(\xi_t)_{t\in[0,T]}$ be a 
process in ${\cal S}^2$, supposed to be \emph{right u.s.c.}
Let 
$(Y,Z,k,h,A,C)$ be the solution to the reflected BSDE with parameters $(f,\xi)$ as in Definition \ref{def_solution_RBSDE}.
\begin{description}
\item[$\bullet$] 
For each stopping time $S$ $\in$ $\T_0$, we have \footnote{ In other words, the process $(Y_t)$ aggregates the value family $(V(S), S \in \T_0)$ defined by \eqref{vvv}, that is $Y_S= V(S)$ a.s. for all $S$ $\in$ $\stopo$.}
 \begin{equation}\label{prixam}
Y_S = {\rm ess} \sup_{\tau \in \T_{S,T}}{\cal E}^f_{S, \tau}(\xi_{\tau})\,\, \quad   \rm{a.s.}
\end{equation}
\item[$\bullet$]  Moreover, the  stopping time 
$\tau_S^{\varepsilon}$  defined by \eqref{definitiontau},  that is, $\tau^{\varepsilon}_S  = \inf \{ t \geq S,\,\, Y_t \leq \xi_t + \varepsilon\}$,
  
  satisfies
 \begin{equation}\label{novelbis}
 Y_S \,\,\leq \,\,  {\cal E}^f_{S, \tau^{\varepsilon}_S}(\xi_{\tau^{\varepsilon}_S})
 + L 
 \varepsilon \quad   \rm{a.s.}\,,
 \end{equation}
where $L$ is a constant which only depends on $T$ and the Lipschitz constant  $K$ of $f$.

  In other words, $\tau_S^{\varepsilon}$ is an $L\varepsilon$-optimal stopping time 
  for  problem \eqref{prixam}.

\end{description}
\end{Theorem}
\begin{proof} The arguments are classical.
Let us show the inequality \eqref{novelbis}. 
Since by Lemma \ref{tol}, the process $(Y_t)$ is 
 an ${\cal E}^{f}$-martingale on $[S, \tau_S^{\varepsilon}]$, we get 
 $Y_S= {\cal E}^f_{S, \tau_S^{\varepsilon}}(Y_{\tau_S^{\varepsilon}})$ a.s.\,
Since $\xi$ is right u.s.c.\,, we can apply Lemma \ref{lemma_epsilon_optimality}.  Using this, the  monotonicity property of the conditional $f$-expectation and  the {\em a priori estimates} for BSDEs (cf. \cite{QuenSul} which still hold in our case of a general filtration),
we derive that 
\begin{equation*} \label{prem}
  Y_S =  {\cal E}^f_{S, \tau^{\varepsilon}_S}(Y_{\tau^{\varepsilon}_S})
    \leq  {\cal E}^f_{S, \tau^{\varepsilon}_S}(\xi_{\tau^{\varepsilon}_S}+ \varepsilon)  \leq 
     {\cal E}^f_{S, \tau^{\varepsilon}_S}(\xi_{\tau^{\varepsilon}_S})
 + L 
 \varepsilon \quad   \rm{a.s.},
\end{equation*}
%
%
where $L$ is a positive constant depending only on $T$ and the Lipschitz constant $K$ of the driver $f$; this gives the desired inequality \eqref{novelbis}.
 Moreover, as $\varepsilon$ is an arbitrary nonnegative number, we get
$
Y_S \,\, \leq  \,\,{\rm ess} \sup_{\tau \in \T_{S,T}}{\cal E}^f_{S ,\tau}(\xi_{\tau})$  
a.s.\,\, 

It remains to show the converse inequality.
Let $\tau$ $\in$ $\T_{S,T}$. 
By Lemma \ref{compref} 
in the Appendix,
 the process $(Y_t)$ is a strong ${\cal E}^f$-supermartingale.  
Hence, for each $\tau \in \T_{S,T}$, we have 
$Y_{S} \geq {\cal E}^f_{S ,\tau}(Y_{\tau})\geq {\cal E}^f_{S ,\tau}(\xi_{\tau})  \quad\text{ a.s.}\,, $ 
where the second inequality follows from the inequality $Y \geq \xi$ and the monotonicity property of 
${\cal E}^f(\cdot)$ (with respect to terminal condition).
By taking the supremum over $\tau \in \T_{S,T}$, we get
$
Y_S  \geq  {\rm ess} \sup_{\tau \in \T_{S,T}} {\cal E}^f_{S ,\tau}(\xi_{\tau})$ a.s.\,
We thus derive the desired equality \eqref{prixam}, which completes the proof.
\end{proof}

We  now investigate the question of the existence of optimal stopping times for the optimal stopping  problem \eqref{prixam}. We first provide an optimality criterion.

\begin{Lemma}[Optimality criterion] \label{optcri}
Let $(\xi_t, 0 
\leq t \leq T )$ be a 
 process  \footnote{Let us emphasize  that this optimality criterion holds true without an  assumption of right-upppersemicontinuity of the process $\xi$.} in ${\cal S}^2$ and let $f$ be a predictable Lipschitz driver satisfying Assumption \ref{Royer}. 
Let $S \in \stopo$ and  $\tau^* \in  \T_{S,T}.$ 
If $Y$ is a strong ${\cal E}^f$-martingale on $[S, \tau^*]$ with $Y_{ \tau^*}= \xi_{\tau^*}$ a.s., then the stopping time $\tau^*$ is optimal at time $S$  (i.e. $Y_S =  {\cal E}^f_{S,  \tau^*}(\xi_{\tau^*})$ a.s.). The converse statement also holds true,  if, in addition, the  inequality from  Assumption \ref{Royer} is strict 
(that is,   $\theta_t^{y,z, \mathpzc{k}_1,\mathpzc{k}_2} > -1$).   
\end{Lemma}

\begin{proof}
The proof of this result in the case of a Brownian-Poisson filtration can be found in \cite[Proposition 4.1 ]{MG}. The proof in our case of a general filtration is identical and is therefore  omitted. 
\end{proof}
We now show that if $\xi$ is assumed to be r.u.s.c. and also l.u.s.c. along stopping times, then there exists an optimal stopping time. 

Let $S\in\T_0$. Let us  recall the definition of $\tau_S^\varepsilon$ from before: 
 \begin{equation*}
 \tau^{\varepsilon}_S:= \inf\{t\geq S \,,  { Y}_t\leq \xi_t+\varepsilon \}.
 \end{equation*} 
 We notice that $\tau^{\varepsilon}_S$ is non-increasing in $\varepsilon$. Let $(\varepsilon_n)$ be a non-increasing positive sequence converging to $0$.  We set 
 $$\hat\tau_S:=\lim_{n\to\infty}\uparrow \tau_S^{\varepsilon_n}.$$
 The random time $\hat\tau_S$ is a stopping time in $\T_S$.\\
  We also set $$\tau_S^0:=\inf\{t\geq S \,,  { Y}_t= \xi_t \}.$$
  We notice that $ \tau_S^{\varepsilon_n}\leq \tau_S^0$ a.s.  for all $n$. Hence, by passing to the limit, we get $\hat\tau_S\leq \tau_S^0$ a.s.

In the following theorem we show that, under the additional assumption  that $\xi$ is  \emph{l.u.s.c. along stopping times}, the stopping time $\hat\tau_S$ is an\emph{ optimal} stopping time at time $S$.  We also show that the stopping times $\hat \tau_S$ and $\tau^0_S$ coincide.
\begin{Theorem}[Existence of optimal stopping time]\label{Thm_existence_optimal}
Let $(\xi_t, 0 
\leq t \leq T )$ be an r.u.s.c.  process in ${\cal S}^2$ and let $f$ be a predictable Lipschitz driver satisfying Assumption \ref{Royer}. We assume, in addition, that $(\xi_t)$ is \emph{l.u.s.c. along stopping times.} Then,  
the  stopping time $\hat\tau_S$ is $S$-{\em optimal}, in the sense that it attains the supremum in \eqref{prixam}. Moreover,  $\hat\tau_S= \tau_S^0$ a.s.

\end{Theorem}

\begin{proof}
As $(\xi_t)$ is l.u.s.c. along stopping times,  
 we have
\begin{equation}\label{eq_limsup}
\limsup_{n\to\infty}\xi_{\tau_S^{\varepsilon_n}}\leq \xi_{\hat\tau_S} \text{ a.s. }
\end{equation}
 By applying Fatou's lemma for (non-reflected) BSDEs (cf. Lemma A.5 in \cite{DQS5} \footnote{Note that Fatou's lemma for (non-reflected) BSDEs, shown in \cite{DQS5} in the case of a Brownian-Poisson filtration, still holds true  in our framework of a general filtration.}),  we obtain
\begin{equation}\label{eq_limsup2}
\limsup_{n\to\infty} \mathcal{E}^f_{S, \tau_S^{\varepsilon_n}} \big(\xi_{\tau_S^{\varepsilon_n}}\big)\leq \mathcal{E}^f_{S, \hat\tau_S} \big(\limsup_{n\to\infty} \xi_{\tau_S^{\varepsilon_n}}\big)\leq \mathcal{E}^f_{S, \hat\tau_S} \big( \xi_{\hat \tau_S}\big)  \text{ a.s., }
\end{equation}
where the last inequality follows from \eqref{eq_limsup} and from the monotonicity of $\mathcal{E}^f_{S, \hat\tau_S}(\cdot).$ On the other hand, from Eq. \eqref{novelbis} in Theorem \ref{caracterisation}, we have $Y_S\leq \limsup_{n\to\infty} \mathcal{E}^f_{S, \tau_S^{\varepsilon_n}} \big(\xi_{\tau_S^{\varepsilon_n}}\big)$ a.s. From this, together with  \eqref{eq_limsup2}, we get 
$Y_S\leq  \mathcal{E}^f_{S, \hat\tau_S} \big( \xi_{\hat \tau_S}\big)  \text{ a.s., }$ which shows that $\hat\tau_S$ is an optimal stopping time. \\
Let us now prove the equality $\hat\tau_S= \tau_S^1$ a.s. We have already noticed that $\hat\tau_S\leq \tau_S^1$ a.s.  It remains to show the converse inequality. 
 Note that for each $S \in\stopo$, $Y_S$ is equal a.s. to the value  at time $S$ of the linear optimal stopping problem associated with the pay-off process $(\xi_t)$ and the instantaneous  reward process $(\bar f_t)$ defined by $\bar f_t(\omega, t): = f(\omega, t, Y_{t-}(\omega),Z_t(\omega),k_t(\omega))$, that is \begin{eqnarray}\label{eq_123}
 Y_S= \esssup_{\tau \in \stops} E[ \xi_{\tau} + \int_S^\tau \bar f_u du \mid \Fc_S] \text{ a.s.}.
\end{eqnarray}  It is not difficult to see that $\hat\tau_S$ is also optimal for this linear optimal stopping problem. Now, from classical   results on  linear optimal stopping, $\tau^0_S$ is the minimal optimal stopping time for  problem \eqref{eq_123}; hence, we have  $\hat \tau_S \geq \tau^0_S$ a.s., which completes the proof. 

\end{proof}

\begin{Proposition}\label{Thm_continu} Let $(\xi_t, 0 
\leq t \leq T )$ be an r.u.s.c.  process in ${\cal S}^2$ and let $f$ be a predictable Lipschitz driver.  We assume, in addition, that $(\xi_t)$ is \emph{l.u.s.c. along stopping times.} Let $(Y,Z,k,h,A,C)$ be the solution to the reflected BSDE with parameters $(f,\xi)$ as in Definition \ref{def_solution_RBSDE}. Then, the process $A$ is continuous. 
\end{Proposition}
\begin{proof}
Given the solution $(Y,Z,k,h,A,C)$  to the reflected BSDE with parameters $(f,\xi)$, we define the process $\bar f$ by
$$\bar{f}(\omega, t):=f(\omega,t, Y_{t-}(\omega), Z_t(\omega),k_t(\omega)).$$ 
The process $\bar f$ is a predictable process in $\H^2$. From the definition of $\bar f$ and from Definition \ref{def_solution_RBSDE}, we see that $(Y,Z,k,h,A,C)$ is the solution of the RBSDE with driver process $\bar f$ and obstacle $\xi$. By Theorem \ref{rf}  (on RBSDEs  with given \emph{driver} \emph{process} and  linear optimal stopping), we have that, for all $S\in\T_0$,
\begin{eqnarray}\label{eq_new0}
 Y_S= \esssup_{\tau \in \stops} E[ \xi_{\tau} + \int_S^\tau \bar f_u du \mid \Fc_S] \text{ a.s.},
\end{eqnarray} 
which is equivalent to 
$Y_S+\int_0^S \bar f_u du= \esssup_{\tau \in \stops} E[ \xi_{\tau} + \int_0^\tau \bar f_u du \mid \Fc_S] \text{ a.s.}$

From results on classical optimal stopping with linear expectations, we deduce that $A$ is continuous, as $(\xi_t)$  is r.u.s.c. and l.u.s.c. along stopping times (cf., e.g.,  Proposition B.10 in \cite{Kob} \footnote{Note that Proposition B.10 in \cite{Kob} also holds true in the case where the reward process is not necessarily nonnegative.}). 
\end{proof}
\section{${\cal E}^f$-Mertens decomposition of strong ${\cal E}^f$-supermartingales with respect to  a general filtration} \label{Mertens2}
By using the above characterization of the solution of the RBSDE with an r.u.s.c. obstacle as the value function of the non-linear optimal stopping problem \eqref{vvv} (cf. Theorem \ref{caracterisation}), we derive  an ${\cal E}^f$-Mertens decomposition of strong ${\cal E}^f$-supermartingales, which generalizes the one provided in \cite{MG} 
(cf. Theorem 5.2 in \cite{MG}) to the case of a general filtration.\footnote{An ${\cal E}^f$-Mertens decomposition was also shown in \cite{Bouchard} (at the same time as in \cite{MG}) in the case of a driver 
$f(t,y,z)$ which does not depend on $\mathpzc{k}$ by using a different approach.}

As mentioned before, this  is an important property in the present work which will allow us to address  the non-linear optimal stopping problem in the completely irregular case 
(cf. Section \ref{subsec_Ref}, more precisely the proof of Proposition \ref{Properties_Ref}, and also Theorem \ref{caranonlinear}).

\begin{Theorem}[${\cal E}^f$-Mertens decomposition]\label{calmertens}
Let $(Y_t)$ be a process  in ${\cal S}^2$. Let $f$ be a  Lipschitz driver satisfying Assumption \ref{Royer}.
The process  $(Y_t)$ is a strong ${\cal E}^f$-supermartingale 
 if and only if 
there exists a nondecreasing 
 right-continuous predictable process $A$ in ${\cal S}^2$ with $A_0=0$ and 
a nondecreasing 
 right-continuous adapted purely discontinuous process $C$ in ${\cal S}^2$ with $C_{0-}=0$, 
 as well as  three processes $Z \in   \H^2 $, $k \in  \mathbb{H}^2_\nu$ and $h \in {\cal M}^{2, \bot}$,  such that 
  %
  %
a.s. for all $t \in [0,T]$,
  %
  %
 \begin{equation}\label{cmertens}
 -d  Y_t  \displaystyle =  f(t,Y_{t}, Z_{t}, k_t)dt + dA_t + d C_{t-}-  Z_t dW_t-\int_E k_t(e)\tilde{N}(dt,de) - dh_t, \quad 0\leq t\leq T.
\end{equation}
This decomposition is unique. Moreover, a strong $\mathcal{E}^f$-supermartingale is necessarily r.u.s.c. 
\end{Theorem}
\begin{proof} 
 Assume that $(Y_t)$ is a strong ${\cal E}^f$-supermartingale.  By the same arguments as in \cite{MG} (cf. Lemma 5.1 in \cite{MG}), it can be shown that the process $(Y_t)$ is  \emph{r.u.s.c.}
 Let $S$ $\in$ $\T_0$.
Since $(Y_t)$ is a strong ${\cal E}^f$-supermartingale, we derive that for all $\tau \in \T_S$, we have $Y_{S} \geq {\cal E}^f_{S ,\tau}(Y_{\tau})$ a.s.\, We  get
  $Y_S  \geq  {\rm ess} \sup_{\tau \in \T_S} {\cal E}^f_{S ,\tau}(Y_{\tau})$ a.s.
  Now, by definition of the essential supremum, $Y_S  \leq  {\rm ess} \sup_{\tau \in \T_S} {\cal E}^f_{S ,\tau}(Y_{\tau}) $ a.s.\, because $S \in \T_S$. 
 Hence, 
  $ Y_S  =  {\rm ess} \sup_{\tau \in \T_S} {\cal E}^f_{S ,\tau}(Y_{\tau})$ a.s.\,By Theorem \ref{caracterisation}, the process $(Y_t)$ coincides with the solution of the reflected BSDE associated with the (r.u.s.c.) obstacle $(Y_t)$, and thus admits the decomposition \eqref{cmertens}.\\
  The converse follows from Lemma \ref{compref} in the Appendix.
%
%
  \end{proof}


\section{Optimal stopping with non-linear $f$-expectation in the completely irregular case: the direct part of the approach}\label{irregularcase}
We now  turn to the study of the non-linear optimal stopping  problem \eqref{vvv} in the more difficult case where $(\xi_t)$ is \textit{completely irregular}. Since the process $(\xi_t)$ is not r.u.s.c.\,, 
 the inequality $ { Y}_{ \tau_S^{\varepsilon} }\leq \xi_{  \tau_S^{\varepsilon}}+\varepsilon$ (i.e. inequality \eqref{lambdabis}) does not necessarily hold (not even in the simplest case of linear expectations; cf., e.g.,  \cite{EK}).  
This prevents us from adopting here the approach used in the r.u.s.c. case to prove an infinitesimal  characterization of the  value  of the non-linear optimal stopping problem  in terms of the solution of an RBSDE.   Thus, when $\xi$ is completely irregular, we  have to proceed differently. We use  a combined approach which consists in a direct part and an RBSDE-part.  
This section is devoted to  the   \emph{ direct part} of our approach to the  non-linear optimal stopping problem \eqref{vvv}. 

\subsection{Preliminary results on the value family}
Let us first introduce the definition of an admissible family of random variables indexed by stopping times in $\stopo$ (or $\stopo$-system in the vocabulary of Dellacherie and Lenglart \cite{DelLen2}). 

\begin{Definition}\label{def.admi}
We say that a family 
$U=(U(\tau), \, \tau \in \stopo)$  is \emph{admissible} if it satisfies the following conditions 
\par
1. \quad for all
$\tau \in \stopo$, $U(\tau)$ is a real-valued  $\mathcal{F}_\tau$-measurable random variable. \par
 2. \quad  for all
$\tau,\tau'\in \stopo$, $U(\tau)=U(\tau')$ a.s.  on
$\{\tau=\tau'\}$.

Moreover, we say that an admissible family 
$U$  is \emph{square-integrable}  if for all
$\tau \in \stopo$, $U(\tau)$  is square-integrable. 
\end{Definition}
\begin{Lemma}[\emph{Admissibility of the family $V$}]\label{P1.Adm}
The family $V=(V(S), S\in \stopo)$ defined in \eqref{vvv} is a square-integrable admissible family.
\end{Lemma}
\begin{proof} 
The proof uses arguments similar to those used  in the "classical" case of linear expectations (cf., e.g., \cite{KQR}), combined with some properties of $f$-expectations.\\
 For each $S\in\stopo$,  V(S) is an $\cf_S$-measurable  square-integrable random variable, due to  the definitions of the conditional $f$-expectation and  of the essential supremum (cf.  \cite{Neveu}). 
Let us prove Property 2 of the  definition of admissibility. Let $S$ and $S'$ be two stopping times in $\stopo$. We set $A:=\{S=S'\}$ and we show that $V(S)=V(S')$, $P$-a.s. on $A$. 
For each $\tau\in\stops$, we set  $\tau_A:= \tau{\bf 1}_A+T{\bf 1}_{A^c}$. We have $\tau_A\geq S'$ a.s.
By using  the fact that $S=S'$ a.s. on $A$,  the fact that $\tau_A=\tau$ a.s. on $A$, and a standard property of conditional $f$-expectations (cf., e.g., Proposition  A.3 in \cite{MG2} which  can be extended without difficulty to the framework of general filtration), we obtain 
\begin{equation*}
\begin{aligned}
{\bf 1}_A\ce^{f}_{S,\tau}[\xi_\tau]&={\bf 1}_A\ce^{f}_{S',\tau}[\xi_\tau]= 
\ce^{f^\tau{\bf 1}_A}_{S',T}[\xi_\tau{\bf 1}_A]=
\ce^{f^{\tau_A}{\bf 1}_A}_{S',T}[\xi_{\tau_A}{\bf 1}_A]=
{\bf 1}_A\ce^{f}_{S',\tau_A}[\xi_{\tau_A}]\leq {\bf 1}_AV(S'),
\end{aligned}
\end{equation*}
 where  $f^{\tau} (t,y,z,\mathpzc{k}):= f(t,y,z,\mathpzc{k}){\bf 1}_{\{t\leq \tau\}}.$
By taking the $\esssup$ over $\stops$ on both sides,  we get 
${\bf 1}_AV(S)\leq {\bf 1}_AV(S').$ We obtain the converse inequality by interchanging the roles of $S$ and $S'$.

\end{proof}

\begin{Lemma}[Optimizing sequence]\label{P1.2a}
For each $S\in\stopo$, there exists a sequence $(\tau_n)_{n \in \mathbb{N}}$ of stopping times in $\stops$ 
such that the sequence $(
 {\cal E}^{f}_{S,\tau_n} (\xi_{\tau_n}))_{n \in \mathbb{N}}$ is nondecreasing and\\
$V(S)  = \lim_{n \to \infty} \uparrow {\cal E}^{f}_{S,\tau_n} (\xi_{\tau_n})$ a.s.
\end{Lemma}
\begin{proof}
Due to a classical result on essential suprema (cf. \cite{Neveu}), it is sufficient to show that,
for each $S\in\stopo$, 
the family  $({\cal E}^{}_{S,\tau} (\xi_\tau), \; \tau \in\stops)$ 
is stable under pairwise maximization. 
Let us fix $S\in\stopo$. Let $\tau\in\stops$ and $\tau'\in\stops$.  
We define  $A:=\{\, {\cal E}^{f}_{S,\tau'} (\xi_{\tau'})\leq 
 {\cal E}^{f}_{S,\tau} (\xi_\tau)\,\}$ and
 $\nu:=\tau {\bf 1} _A+\tau' {\bf 1} _{A^c}$. We have  $A \in \mathcal{F}_S$ and $ \nu\in\stops$.
 We compute $ {\bf 1} _A{\cal E}^{f}_{S,\nu} (\xi_\nu)= $ 
 $ {\cal E}^{ f^{\nu}{\bf 1} _A }_{S,T} (\xi_\nu {\bf 1} _A)= {\cal E}^{ f^{ \tau}{\bf 1} _A }_{S,T} (\xi_\tau {\bf 1} _A)
$ $=$ $ {\bf 1} _A{\cal E}^{ f}_{S,\tau} (\xi_\tau)$ a.s. \,  Similarly, we show  $ {\bf 1} _{A^c}{\cal E}^{f}_{S,\nu} (\xi_\nu)={\bf 1} _{A^c}{\cal E}^{f}_{S,\tau'} (\xi_{\tau'})$.  It follows that
 $ {\cal E}^{f}_{S,\nu} (\xi_\nu)=$  $ {\cal E}^{f}_{S,\tau} (\xi_\tau){\bf 1} _A+
 {\cal E}^{f}_{S,\tau'}(\xi_{\tau'}){\bf 1} _{A^c}=$
 $ {\cal E}^{f}_{S,\tau} (\xi_\tau)\vee $ $ 
 {\cal E}^{ f}_{S,\tau'} (\xi_{\tau'})$,
which shows the stability under pairwise maximization and concludes the proof. 
\end{proof}
\begin{Definition}[${\cal E}^{f}$-\emph{supermartingale family}]
An admissible square-integrable family $U:=(U(S), \; S\in\stopo)$ is said to be an  ${\cal E}^{f}$-\emph{supermartingale family}  if for all 
$S, S^{'}$ $ \in\stopo$ such that $S \leq S^{'}$ a.s., 
$
{\cal E}^{f}_{S,S'} (U(S'))  \leq  U(S) $ a.s.
\end{Definition}

\begin{Definition}[\emph{Right-uppersemicontinuous family}]\label{Def_rusc}
An admissible family $U:=$ $(U(S), \; S\in\stopo)$ is said to be  a \emph{right-uppersemicontinuous} (along stopping times) \emph{family} if, for any  $(\tau_n)$ nonincreasing sequence in $\stopo$  and any $\tau$ in $\stopo$
such that $\tau=\lim \downarrow \tau_n$, we
have 
$U(\tau)\geq \limsup_{n\to\infty} U(\tau_n)$ a.s. 
\end{Definition}

\begin{Lemma}\label{Prop_right}
Let $U:=(U(S), \; S\in\stopo)$ be an   ${\cal E}^{f}$-supermartingale family. Then, $(U(S), \; S\in\stopo)$ is a right-uppersemicontinuous (along stopping times) family.
\end{Lemma}
\begin{proof}
Let $\tau \in \stopo$ and let $(\tau_n)\in \stopo^{\N}$ be a nonincreasing sequence of stopping times such that $\lim_{n \rightarrow + \infty} \tau_n = \tau$ a.s. and  for all $n\in\N$, $\tau_n > \tau$ a.s. on 
$\{\tau < T\}$, and such that $\lim_{n \rightarrow + \infty} U(\tau_n)$ exists a.s.\,
As $U$  is an  ${\cal E}^{f}$-supermartingale 
family and as the sequence $(\tau_n)$ is nonincreasing, we have 
${\cal E}^{f}_{\tau,\tau_n} (U(\tau_n))  \leq {\cal E}^{f}_{\tau,\tau_{n+1}} (U(\tau_{n+1}))\leq U(\tau)$ a.s. 
Hence, the sequence $({\cal E}^{f}_{\tau,\tau_n} (U(\tau_n)))_n$ is nondecreasing and 
$U(\tau)\geq \lim \uparrow {\cal E}^{f}_{\tau,\tau_n} (U(\tau_n)). $ This inequality, combined with the property of continuity of  BSDEs with respect to terminal time and terminal condition (cf.   \cite[Prop. A.6]{QuenSul} which still holds in the case of a general filtration) gives
$$U({\tau}) \geq \lim_{n \rightarrow + \infty}  {\cal E}^f_{\tau ,\tau_n}(U({\tau_n}))= {\cal E}^f_{\tau ,\tau}(\lim_{n \rightarrow + \infty} U({\tau_n}))= \lim_{n \rightarrow + \infty} U({\tau_n})\quad {\rm a.s.}$$ 
By Lemma 5 of Dellacherie and Lenglart \cite{DelLen2} \footnote{The chronology $\Theta$ (in the vocabulary and notation of \cite{DelLen2}) which we work with here is the chronology of all stopping times, that is,  $\Theta=\stopo$; hence $[\Theta]=\Theta=\stopo$.}, the family $(U(S))$ is thus right-uppersemicontinuous (along stopping times).
\end{proof}

\begin{Theorem}\label{Prop_sup} The value family $V=(V(S), \; S\in\stopo)$ defined in \eqref{vvv}  is an  ${\cal E}^{f}$-supermartingale family. In particular, $V=(V(S), \; S\in\stopo)$ is a right-uppersemicontinuous (along stopping times) family
 in the sense of Definition \ref{Def_rusc}.
\end{Theorem}

\begin{proof}
We know from Lemma \ref{P1.Adm} that $V=(V(S), S\in \stopo)$ is a square-integrable admissible family.
Let $S\in \stopo$ and $S'\in\stops$.  We will  show that
${\cal E}^{f}_{S,S'}(V(S'))\leq  V(S)$ a.s., which will prove that $V$ is an  ${\cal E}^{f}$-supermartingale family.  
By Lemma \ref{P1.2a}, there exists a sequence $(\tau_n)_{n \in \mathbb{N}}$ of stopping times such that $\tau_n\geq S'$ a.s. and  $V(S')  = \lim_{n \to \infty} \uparrow {\cal E}^{f}_{S',\tau_n} (\xi_{\tau_n}) \quad
\mbox{\rm a.s.}$ By using this equality, the property of continuity of BSDEs, and the consistency of conditional $f$-expectation, we get
$${\cal E}^{f}_{S,S'}(V(S'))  = {\cal E}^{f}_{S,S'}(\lim_{n \to \infty} \uparrow {\cal E}^{f}_{S',\tau_n} (\xi_{\tau_n}))=\lim_{n \to \infty}{\cal E}^{f}_{S,S'}({\cal E}^{f}_{S',\tau_n} (\xi_{\tau_n}))=\lim_{n \to \infty}{\cal E}^{f}_{S,\tau_n}(\xi_{\tau_n})\leq V(S). $$
Hence, $V$ is an   ${\cal E}^{f}$-supermartingale family. This property, together with Lemma \ref{Prop_right}, yields that $V$ is a  right-uppersemicontinuous (along stopping times) family. 
\end{proof}
\subsection{Aggregation and Snell characterization}
Using the above results on the value family $V=(V(S), \; S\in\stopo)$, we  show the following theorem, which generalizes some results of classical optimal stopping theory (more precisely, the assertion (i) from Lemma 
\ref{pro}) to the case of an optimal stopping problem with $f$-expectation.

\begin{Theorem}[Aggregation and Snell characterization]\label{Prop_agrege}
There exists a unique  right- uppersemicontinuous optional process, denoted by $(V_t)_{t\in[0,T]}$, which aggregates the value family $V=(V(S), \; S\in\stopo)$. Moreover, $(V_t)_{t\in[0,T]}$ is the  ${\cal E}^f$-{\em Snell envelope} of the pay-off process $\xi$, 
that is, the smallest strong 
${\cal E}^f$-supermartingale greater than or equal to $\xi$.
\end{Theorem}

\begin{proof}
By Theorem \ref{Prop_sup},     the value family $V=(V(S), \; S\in\stopo)$ is  a right-uppersemicontinuous family (or a right-uppersemicontinuous  $\stopo$-system  in the vocabulary of Dellacherie-Lenglart \cite{DelLen2}).  Applying Theorem 4 of  Dellacherie-Lenglart (\cite{DelLen2}),  gives the existence of a unique (up to indistinguishability) right-uppersemicontinuous optional  process $(V_t)_{t\in[0,T]}$ which \emph{aggregates} the value family $(V(S), \; S\in\stopo)$. 
From this aggregation property, namely the property  $V_S=V(S)$ a.s.  for each $S\in\stopo$, and from Theorem \ref{Prop_sup}, we deduce that the process $(V_t)_{t\in[0,T]}$ is a strong 
${\cal E}^f$-supermartingale. Moreover, we have $V_S=V(S) \geq \xi_S$ a.s. for each $S\in\stopo$, which implies that
$V_t\geq \xi_t$, for all $t\in[0,T]$, a.s. \\
Let us now prove that  the process $(V_t)_{t\in[0,T]}$ is \emph{the smallest}  strong 
${\cal E}^f$-supermartingale greater than or equal to $\xi$. Let $(V'_t)_{t\in[0,T]}$  be a strong 
${\cal E}^f$-supermartingale such that $V'_t\geq \xi_t$, for all $t\in[0,T]$, a.s. Let $S\in\stopo$. We have 
$V'_\tau\geq \xi_\tau$ a.s. for all $\tau\in\stops$. Hence, ${\cal E}^{f}_{S,\tau}(V'_{\tau})\geq {\cal E}^{f}_{S,\tau}(\xi_{\tau})$ a.s., where we have used the monotonicity of the conditional $f$-expectation. On the other hand, by using the   strong ${\cal E}^f$-supermartingale property of the process $(V'_t)_{t\in[0,T]}$, we have $V'_S\geq {\cal E}^{f}_{S,\tau}(V'_{\tau})$ a.s.   for all $\tau\in\stops$. Hence, $V'_S\geq {\cal E}^{f}_{S,\tau}(\xi_{\tau})$ a.s. for all $\tau\in\stops$. By taking the essential supremum over $\tau\in\stops$ in the  inequality, we get $V'_S\geq \esssup_{\tau\in\stops} {\cal E}^{f}_{S,\tau}(\xi_{\tau})=V(S)= V_S$ a.s. 
Hence, for all $S\in\stopo$, we have $V'_S\geq V_S$ a.s., which yields that 
$V'_t\geq V_t$, for all $t\in[0,T]$, a.s.
The proof is thus complete.
\end{proof}

\section{Non-linear Reflected BSDE with completely irregular obstacle and general filtration: useful properties}\label{sec4}
Our aim now is  
 to establish an infinitesimal characterization for the non-linear problem \eqref{vvv}    in terms of the solution of a non-linear RBSDE (thus generalizing  Theorem  \ref{rf} from the classical linear case to the non-linear case).  In order to do so, we need to establish first some results on non-linear  RBSDEs with \emph{completely irregular} obstacles, in particular, a comparison result for such RBSDEs.  This section is devoted to these results (this is the RBSDE-part of our approach to problem \eqref{vvv}).    The results from this section  extend and complete our work from \cite{MG}, where an assumption of right-uppersemicontinuity on the obstacle is made.   
 Let us note that  the proof of the comparison theorem  from \cite{MG} cannot be adapted to the completely irregular framework considered here; instead, we rely on a Tanaka-type formula for strong (irregular) semimartingales which we also establish.  
 
 \begin{Remark}\label{Rk_direct_approach}{\emph{(A "bottle-neck" of the direct approach)}}
 One might wonder whether  the infinitesimal characterization for  the non-linear optimal stopping problem \eqref{vvv} can be obtained by pursuing the direct study of the value process $(V_t)$ of problem \eqref{vvv}, similarly to what was done in the classical linear case in  Sub-section \ref{subsection_linear}. In the classical case,  we applied  Mertens decomposition for $(V_t)$; then,  we showed directly the minimality properties for the processes   $A^d$ and  $A^c$ (cf. Lemmas \ref{propertiesvalue_prel} and  \ref{propertiesvalue_Ac_prel}) by using the martingale property on the interval $[S, \tau_S^{\lambda}]$ from Lemma \ref{pro_prel}(iii), which itself relies on Maingueneau's penalization approach (cf. also Remarks \ref{Rk_minimal} and \ref{Rk_Maingueneau}). In the non-linear case, Mertens decomposition is generalized   by the $\mathcal{E}^f$-Mertens decomposition 
(cf. Theorem \ref{calmertens}).
 However, the  analogue in the non-linear case of the martingale property  of Lemma \ref{pro}[(iii)] (namely,  the $\mathcal{E}^f$-martingale property) cannot be obtained via Maingueneau's approach (not even in the case of nonnegative $\xi$ and under the additional assumption $f(t,0,0,0)=0$ which ensures the non-negativity of $\mathcal{E}^f$)  due to the \emph{lack of convexity} of the functional  $\mathcal{E}^f$. 
 \end{Remark}

\subsection{Tanaka-type formula}\label{subsect_comparison}

The following lemma will be used in the proof of the comparison theorem for RBSDEs with irregular obstacles.
The lemma can be seen as an extension of Theorem 66 of \cite[Chapter IV]{Protter} from the case of right-continuous semimartingales to the more general case of strong optional semimartingales.

\begin{Lemma}[Tanaka-type formula]\label{Lemma_Protter}
Let $X$ be a (real-valued) strong optional semimartingale with decomposition 
$X=X_0+M+A+B$, where $M$ is a local (cadlag) martingale, $A$ is a right-continuous adapted process of finite variation such that $A_0=0$, $B$ is a  left-continuous adapted purely discontinuous process of finite variation such that  $B_0=0$. Let $f:\R\longrightarrow \R$ be a convex function. Then, $f(X)$ is a strong optional semimartingale. Moreover, denoting by $f'$ the left-hand derivative of the convex function $f$, we have  

$$f(X_t)=f(X_0)+\int_{]0,t]} f'(X_{s-})d(A_s+M_s) 
+\int_{[0,t[} f'(X_{s})dB_{s+}+K_t,$$
where $K$ is a nondecreasing adapted process (which is in general neither left-continuous nor right-continuous)
 such that
$$ \Delta K_t= f(X_t)-f(X_{t-})- f'(X_{t-}) \Delta X_t \text{ and }
 \Delta_+ K_t= f(X_{t+})-f(X_{t})- f'(X_{t}) \Delta_+ X_t.$$
\end{Lemma}

\begin{proof}
Our proof follows the proof of Theorem 66 of \cite[Chapter IV]{Protter} with suitable changes.\\
{\em Step 1.} We assume that $X$ is bounded; more precisely, we assume that there exists $N\in\N$ such that $|X|\leq N$. We know (cf. \cite{Protter}) that there exists a sequence $(f_n)$ of  twice continuously differentiable convex  functions such that $(f_n)$ converges to $f$,  and $(f'_n)$ converges to  $f'$ from below. By applying Gal'chouk-Lenglart's formula (cf., e.g., Theorem  A.3 in \cite{MG})  to $f_n(X_t)$, we obtain for all $\tau\in\stopo$
\begin{equation}\label{eq_Ito_usuel}
 f_n(X_\tau)=f_n(X_0)+ \int_{]0,\tau]} f'_n(X_{s-})d(A_s+M_s) 
+\int_{[0,\tau[} f'_n(X_{s})dB_{s+}+K^n_\tau, \text{ a.s., where }
 \end{equation} 
\begin{equation}\label{eq_K^n}
 \begin{aligned}
 K_\tau^n:= &\sum_{0<s\leq \tau}\left[ f_n(X_s)-f_n(X_{s-})- f'_n(X_{s-}) \Delta X_s\right]
 +\sum_{0\leq s<\tau}\left[ f_n(X_{s+})-f_n(X_{s})- f'_n(X_{s}) \Delta_+ X_s\right]\\
 &+
 \frac 1 2 \int_{]0,\tau]}  f''_n(X_{s-})d\langle M^{c},M^{c}\rangle_s \;\text{ a.s. }
 \end{aligned}
 \end{equation} 
 We show that $(K_\tau^n)$ is a convergent sequence by showing that the other terms in Equation \eqref{eq_Ito_usuel} converge. The convergence $\int_{]0,\tau]} f'_n(X_{s-})d(A_s+M_s)\underset{n\to\infty}{\longrightarrow} \int_{]0,\tau]} f'(X_{s-})d(A_s+M_s)$  is shown by using the same arguments as in the proof of \cite[Thorem 66, Ch. IV]{Protter}. The convergence of the term  
 $\int_{[0,\tau[} f'_n(X_{s})dB_{s+}$, which is specific to the non-right-continuous case, is shown by using dominated convergence. We conclude that $(K_\tau^n)$ converges and we set $K_\tau:= \lim_{n\to\infty} K_\tau^n$.  
  The process $(K_t)$ is adapted as the limit of adapted processes. Moreover, we have from Eq. \eqref {eq_K^n} and from the convexity of $f_n$ that,  for each $n$, $K_t^n$ is nondecreasing in $t$. Hence, the limit $K_t$ is nondecreasing.  \\
{\em Step 2.}  We treat the general case where $X$ is not necessarily bounded by using a localization argument  similar  to that used in \cite[Th. 66, Ch. IV]{Protter}. 
 \end{proof}

\subsection{Comparison theorem}\label{subsect_RBSDEb}

\begin{Theorem}[Comparison]\label{thmcomprbsde}
 Let $\xi\in\mathcal{S}^2$, $\xi'\in\mathcal{S}^2$  be two 
 processes. Let $f$  and $f'$
 be  Lipschitz drivers   satisfying Assumption~\ref{Royer}. Let $(Y, Z, k, h, A, C)$ (resp.
 $(Y', Z', k', h', A', C')$)  be
the solution of the RBSDE associated with obstacle $\xi$  (resp. $\xi'$) and with driver $f$ (resp. $f'$). 
If $\xi_t\le \xi'_t$, $0\leq t \leq T$ a.s. and 
$f(t,Y'_t, Z'_t, k'_t)\le f'(t,Y'_t, Z'_t, k'_t)$, $0\leq t \leq T$ $dP\otimes dt$-a.s., 
%
then,    
$Y_{t}\le
Y'_{t}, \,0\leq t \leq T \text{ a.s. } $
\end{Theorem}
\begin{proof}
We set $\bar Y_t = Y_t - Y'_t$, $\bar Z_t = Z_t - Z'_t$, $\bar k_t=k_t-k'_t$,  $\bar A_t=A_t - A'_t$, $\bar C_t=C_t - C'_t$, $\bar h_t= h_t-h'_t$,
and $\bar f_t=f(t,Y_{t-},Z_t,k_t)-f'(t,Y'_{t-},Z'_t,k'_t)$.
   Then,\\
  $
  -d \bar Y_t  \displaystyle =  \bar f_t dt + d\bar A_t + d \bar C_{t-}- \bar Z_t dW_t-\int_E \bar k_t(e) \tilde N(dt,de) -d\bar h_t,$ with 
   $\bar Y_{T}  = 0.
   $

Applying Lemma \ref{Lemma_Protter}  to  the positive part of $\bar Y_t$,  
we obtain   
\begin{equation}
\begin{aligned}
\bar Y_t^+=&-\int_{]t,T]}{\bf 1}_{\{ \bar Y_{s-}>0\}} \bar Z_s dW_s-\int_{]t,T]}\int_{E}{\bf 1}_{\{ \bar Y_{s-}>0\}} \bar k_s(e) \tilde N(ds,de)
 - \int_{]t,T]}{\bf 1}_{\{ \bar Y_{s-}>0\}} d \bar h_s\\
&+ \int_{]t,T]}{\bf 1}_{\{ \bar Y_{s-}>0\}} \bar f_s ds +\int_{]t,T]}{\bf 1}_{\{ \bar Y_{s-}>0\}} d\bar A_s+\int_{[t,T[}{\bf 1}_{\{ \bar Y_{s}>0\}} d\bar C_s+ (K_t-K_T). 
\end{aligned}
\end{equation}

We set 
$\delta_t:= \frac{f(t,Y_{t-},Z_t,k_t)-f(t,Y'_{t-},Z_t,k_t)}{ Y_{t-}-Y'_{t-}}{\bf 1}_{\{ \bar Y_{t-}\neq 0\}}$ and 
 $\beta_t:= \frac{f(t,Y'_{t-},Z_t,k_t)-f(t,Y'_{t-},Z'_t,k_t)}{ Z_{t}-Z'_{t}}{\bf 1}_{\{ \bar Z_{t}\neq 0\}}$. 
Due to the Lipschitz-continuity of $f$, the processes $\delta$ and $\beta$ are bounded. 
We note that $\bar f_t=\delta_t\bar Y_t+\beta_t\bar Z_t+f(Y'_{t-}, Z'_t, k_t)-f(Y'_{t-}, Z'_t, k'_t)+ \varphi_t$, where 
$\varphi_t:=  f(Y'_{t-}, Z'_t, k'_t)-f'(Y'_{t-}, Z'_t, k'_t)$. Using this, together with Assumption \ref{Royer}, we obtain
\begin{equation}\label{eq_new_inequality}
\bar f_t \leq  \delta_t \bar Y_t + \beta_t \bar Z_t + \langle \gamma_t\,,\, \bar k_t  \rangle_\nu, + \varphi_t \;\; 0 \leq t \leq T,\quad 
dP \otimes dt-{\rm a.e.},
\end{equation}
where we have set $\gamma_t:= \theta _t^{Y'_{t-}, Z'_t, k'_t, k_t}.$
For $\tau \in {\cal T}_{0,T}$, let $\Gamma_{\tau,\cdot}$ be  the unique solution of the following forward SDE $d \Gamma_{\tau,s}  =  \Gamma_{\tau,s-} [ \delta_s ds + \beta_s d W_s + \int_E\gamma_s(e)\tilde N(ds,de)]$ with initial condition (at the initial time $\tau$)
$\Gamma_{\tau,\tau}  = 1. $
To simplify the notation, we denote  $\Gamma_{\tau,s}$ by $\Gamma_s$ for $s \geq \tau$.

%
By applying  Gal'chouk-Lenglart's formula to the product $(\Gamma_t\bar Y_t^+)$, 
and by using that $\langle h^c, W \rangle =0$,
 we get 
 \begin{equation}\label{eq_product}
\begin{aligned}
&\Gamma_\tau\bar Y_\tau^+=  - (M_\theta  -  M_{\tau}) -
\int_\tau^\theta \Gamma_{s}(\bar Y^+_{s-}\delta_s+\bar Z_s{\bf 1}_{\{ \bar Y_{s-}>0\}}\beta_s -\bar f_s{\bf 1}_{\{ \bar Y_{s-}>0\}})ds\\
&+\int_\tau^\theta \Gamma_{s-}{\bf 1}_{\{ \bar Y_{s-}>0\}}d\bar A_s^c+\sum_{\tau < s\leq \theta} \Gamma_{s-}{\bf 1}_{\{ \bar Y_{s-}>0\}}\Delta\bar A_s-\int_\tau^\theta \Gamma_{s-}dK_s^c-\int_\tau^\theta \Gamma_{s-}dK_s^{d,-}
\\
&+
\int_\tau^\theta \Gamma_{s}{\bf 1}_{\{ \bar Y_{s}>0\}}d\bar C_s-\int_\tau^\theta \Gamma_{s}dK_s^{d,+}
%
 %
-\sum_{\tau < s\leq \theta} \Delta \Gamma_s\Delta \bar Y^+_s. 
\end{aligned}
\end{equation}
where the process $M$ is  defined by $M:= M^W + M^N + M^h$, with 
$M^W_t:=\int_0^t \Gamma_{s-}({\bf 1}_{\{ \bar Y_{s-}>0\}}\bar Z_s+\bar Y^+_{s-}\beta_s) dW_s$,  
and 
$M^N_t:=\int_0^t \int_E \Gamma_{s-}(\bar k_s(e){\bf 1}_{\{ \bar Y_{s-}>0\}}+\bar Y^+_{s-}\gamma_s(e))\tilde N(ds,de)$, and $M^h_t:=\int_0^t  \Gamma_{s^-}{\bf 1}_{\{ \bar Y_{s^-}>0\}}d \bar h_s$. Note that by classical arguments (which use Burkholder-Davis-Gundy inequalities), the stochastic integrals $M^W$, $M^N$ and $M^h$ are martingales. Hence, $M$ is a martingale (equal to zero in expectation).

 %

%
By definition of $\Gamma$, we have
$\Gamma_{\tau}  = 1$, which gives  that $\Gamma_\tau\bar Y_\tau^+= \bar Y_\tau^+$.
Moreover, we have $\int_\tau^\theta \Gamma_{s}{\bf 1}_{\{ \bar Y_{s}>0\}}d\bar C_s=\int_\tau^\theta \Gamma_{s}{\bf 1}_{\{ \bar Y_{s}>0\}}d C_s - \int_\tau^\theta \Gamma_{s}{\bf 1}_{\{ \bar Y_{s}>0\}}d C'_s$. For the first term, it holds\\
  $\int_\tau^\theta \Gamma_{s}{\bf 1}_{\{ \bar Y_{s}>0\}}d C_s=0$. 
 Indeed,  $\{ \bar Y_{s}>0\}=\{ Y_{s}>Y'_{s}\}\subset \{ Y_{s}>\xi_{s}\}$ (as $ Y'_{s}\geq \xi'_{s}\geq \xi_{s}$). This, together with the Skorokhod condition for $C$  gives the equality. For the second term, it holds $- \int_\tau^\theta \Gamma_{s}{\bf 1}_{\{ \bar Y_{s}>0\}}d C'_s\leq 0$, as $\Gamma\geq 0$ and  $d C'$ is a nonnegative measure. Hence, $\int_\tau^\theta \Gamma_{s}{\bf 1}_{\{ \bar Y_{s}>0\}}d\bar C_s\leq 0$. 
 Similarly, we obtain  $\int_\tau^\theta \Gamma_{s-}{\bf 1}_{\{ \bar Y_{s-}>0\}}d\bar A_s^c\leq 0$. Indeed, $\int_\tau^\theta \Gamma_{s-}{\bf 1}_{\{ \bar Y_{s-}>0\}}d\bar A_s^c=\int_\tau^\theta \Gamma_{s-}{\bf 1}_{\{ \bar Y_{s-}>0\}}d A_s^c-\int_\tau^\theta \Gamma_{s-}{\bf 1}_{\{ \bar Y_{s-}>0\}}d A_s^{'c}.$  For the first term, we have $\int_\tau^\theta \Gamma_{s-}{\bf 1}_{\{ \bar Y_{s-}>0\}}d A_s^c=0$.  This is due to the fact that $\{ \bar Y_{s-}>0\}=\{ Y_{s-}>Y'_{s-}\}\subset \{ Y_{s-}>\overline \xi_{s}\}$ (as $ Y'_{s}\geq \xi'_{s}\geq \xi_{s}$, and hence $ Y'_{s-}\geq \overline \xi_{s}$), together with the Skorokhod condition for $A^c$. For the second term, we have $-\int_\tau^\theta \Gamma_{s-}{\bf 1}_{\{ \bar Y_{s-}>0\}}d A_s^{'c}\leq 0.$ We also have  
$-\int_\tau^\theta \Gamma_{s-}dK_s^c\leq 0$ and $-\int_\tau^\theta \Gamma_{s}dK_s^{d,+}\leq 0$. 
Hence, 
\begin{equation}\label{eq_product2}
\begin{aligned}
\bar Y_\tau^+ \leq & - (M_\theta  -  M_{\tau}) -  
\int_\tau^\theta \Gamma_{s}(\bar Y^+_{s-}\delta_s+\bar Z_s{\bf 1}_{\{ \bar Y_{s-}>0\}}\beta_s -\bar f_s{\bf 1}_{\{ \bar Y_{s-}>0\}})ds\\&+\sum_{\tau < s\leq \theta} \Gamma_{s-}{\bf 1}_{\{ \bar Y_{s-}>0\}}\Delta\bar A_s
-\int_\tau^\theta \Gamma_{s-}dK_s^{d,-}
 -\sum_{\tau < s\leq \theta} \Delta \Gamma_s\Delta \bar Y^+_s. 
\end{aligned}
\end{equation}
We compute the last term $\sum_{\tau < s\leq \theta} \Delta \Gamma_s\Delta \bar Y^+_s$.\\
 Let $(p_s)$ be the point process associated with the Poisson random measure $N$ 
 (cf. \cite[VIII Section 2. 67]{DM2}, or \cite[Section III \S d]{J}). 
  We have 
$\Delta \Gamma_s=\Gamma_{s-} \gamma_s (p_s)$ and $\Delta \bar Y^+_s={\bf 1}_{\{ \bar Y_{s-}>0\}}\bar k_s (p_s)-{\bf 1}_{\{ \bar Y_{s-}>0\}}\Delta \bar A_s +\Delta K_s^{d,-}
 + {\bf 1}_{\{ \bar Y_{s-}>0\}}\Delta \bar h_s.
$
Hence, 
\begin{equation}\label{eq_last_term}
\begin{aligned}
&\sum_{\tau < s\leq \theta} \Delta \Gamma_s\Delta \bar Y^+_s= \\
&=\sum_{\tau < s\leq \theta} \Gamma_{s-} {\bf 1}_{\{ \bar Y_{s-}>0\}} \gamma_s(p_s) \bar k_s (p_s)-\sum_{\tau < s\leq \theta}\Gamma_{s-} \gamma_s(p_s)({\bf 1}_{\{ \bar Y_{s-}>0\}}\Delta \bar A_s -\Delta K_s^{d,-}
-{\bf 1}_{\{ \bar Y_{s-}>0\}}\Delta \bar h_s
)
\\
&= \int_{\tau}^\theta \int_{E}\Gamma_{s-} {\bf 1}_{\{ \bar Y_{s-}>0\}} \gamma_s(e) \bar k_s(e)  N(ds,de)-\sum_{\tau < s\leq \theta}\Gamma_{s-} \gamma_s(p_s)({\bf 1}_{\{ \bar Y_{s-}>0\}}\Delta \bar A_s -\Delta K_s^{d,-}
-{\bf 1}_{\{ \bar Y_{s-}>0\}}\Delta \bar h_s
)
\\
&=\int_{\tau}^\theta \int_{E}\Gamma_{s-} {\bf 1}_{\{ \bar Y_{s-}>0\}} \gamma_s(e) \bar k_s(e)  \tilde N(ds,de)+\int_{\tau}^\theta \Gamma_{s-} {\bf 1}_{\{ \bar Y_{s-}>0\}} \langle\gamma_s, \bar k_s\rangle_\nu ds\\
&-\sum_{\tau < s\leq \theta}\Gamma_{s-} {\bf 1}_{\{ \bar Y_{s-}>0\}}\gamma_s(p_s)\Delta \bar A_s+ \sum_{\tau < s\leq \theta}\Gamma_{s-} \gamma_s (p_s)\Delta K_s^{d,-} 
 + \sum_{\tau < s\leq \theta} \Gamma_{s-}{\bf 1}_{\{ \bar Y_{s-}>0\}}
 \gamma_s(p_s)  \Delta \bar h_s.  
\end{aligned}
\end{equation}
By plugging this expression in equation \eqref{eq_product2} and by putting together  the terms in $"ds"$, the terms in $"dK_s^{d,-}"$, and the terms in $"\Delta \bar A_s"$, we get 
\begin{equation}\label{eq_product3}
\begin{aligned}
\bar Y_\tau^+\leq &  - (M_\theta  -  M_{\tau}) 
-  
\int_\tau^\theta \Gamma_{s-}(\bar Y^+_{s-}\delta_s+\bar Z_s{\bf 1}_{\{ \bar Y_{s-}>0\}}\beta_s + {\bf 1}_{\{ \bar Y_{s-}>0\}} \langle\gamma_s, \bar k_s\rangle_\nu-\bar f_s{\bf 1}_{\{ \bar Y_{s-}>0\}})ds\\
&+\sum_{\tau < s\leq \theta} \Gamma_{s-}{\bf 1}_{\{ \bar Y_{s-}>0\}}(1+ \gamma_s(p_s))\Delta \bar A_s^{}-\sum_{\tau < s\leq \theta} \Gamma_{s-}(1+ \gamma_s(p_s))\Delta K_s^{d,-}\\
%
&  -(\tilde M_\theta- \tilde M_{\tau}) -   
\int_\tau ^{\theta} d[ \,\bar h\,, \int_0^\cdot \int_{E}\Gamma_{s-} {\bf 1}_{\{ \bar Y_{s-}>0\}} \gamma_s(e) \tilde N(ds,de)\,]_s, 
\end{aligned}
\end{equation}
where $\tilde M_t:= \int_{0}^t \int_{E}\Gamma_{s-} {\bf 1}_{\{ \bar Y_{s-}>0\}} \gamma_s(e) \bar k_s(e)  \tilde N(ds,de)$. Note that by classical arguments (as for $M$ above), the stochastic integral $\tilde M$ is a martingale, equal to zero in expectation.

We have $-\int_\tau^\theta \Gamma_{s-}(\bar Y^+_{s-}{\bf 1}_{\{ \bar Y_{s-}>0\}}\delta_s+\bar Z_s{\bf 1}_{\{ \bar Y_{s-}>0\}}\beta_s + {\bf 1}_{\{ \bar Y_{s-}>0\}} \langle\gamma_s, \bar k_s\rangle_\nu-\bar f_s{\bf 1}_{\{ \bar Y_{s-}>0\}})ds\leq
\int_\tau^\theta \Gamma_{s-}{\bf 1}_{\{ \bar Y_{s-}>0\}} \varphi_s ds,
$
 due to the inequality \eqref{eq_new_inequality}. The term $-\sum_{\tau < s\leq \theta} \Gamma_{s-}(1+ \gamma_s(p_s))\Delta K_s^{d,-}$ is  nonpositive, as $1+ \gamma_s\geq 0$ by Assumption \ref{Royer}.  The term $\sum_{\tau < s\leq \theta} \Gamma_{s-}{\bf 1}_{\{ \bar Y_{s-}>0\}}(1+ \gamma_s(p_s))\Delta \bar A_s^{}$ is nonpositive, due to $1+ \gamma_s\geq 0$, to the Skorokhod condition for $\Delta A_s^{}$ and to $\Delta A'_s\geq 0$ (the details are similar to those for $d\bar C$ in the reasoning above). 
 Since $\bar h$ $\in$ 
 ${\cal M}^{2,\bot}$, by Remark \ref{2i},
 we derive that  
 the expectation of the last term of the above inequality \eqref{eq_product3} is equal to $0$.  %
 Moreover, the term $\int_\tau^\theta \Gamma_{s-}{\bf 1}_{\{ \bar Y_{s-}>0\}} \varphi_s ds$ is nonpositive, as 
 $\varphi_s=  f(Y'_{s}, Z'_s, k'_s)-f'(Y'_{s}, Z'_s, k'_s) \le 0$ $dP\otimes ds$-a.s. by the assumptions of the theorem.
 We conclude that $E[\bar Y_\tau^+]\leq 0$, which implies $\bar Y_\tau^+= 0$ a.s. The proof is thus complete. 
 \end{proof}
 
 \begin{Remark}\label{Rmk_diff_0}
 Note that due to the irregularity of the obstacles, together with the presence of jumps, we cannot adopt the approaches used up to now  in the literature (see e.g. \cite{ElKaroui97}, \cite{CM}, \cite{QuenSul2} and \cite{MG})  to show the comparison theorem for our RBSDE. 
 \end{Remark}
\subsection{Non-linear operator induced by an RBSDE. 
 Snell characterization}\label{subsec_Ref}
We introduce the non-linear operator $\Ref^f$ (associated with a given non-linear driver $f$) and provide some useful  properties. In particular, we show that this non-linear operator coincides  with the $\mathcal{E}^f$-Snell envelope operator (cf. Theorem \ref{Prop_characterization_Snell}). 

\begin{Definition}[Non-linear operator $\Ref^f$]\label{definitionRef}
Let $f$ be a Lipschitz driver. 
For a    process $(\xi_t)$ $\in$  $\mathcal{S}^2$,  we denote by $\Ref^f[\xi]$ the first component of the solution to the Reflected BSDE with (lower) barrier $\xi$ and with Lipschitz driver $f$. 
\end{Definition}
The operator $\Ref^f[\cdot]$ is well-defined due to Theorem \ref{rexiuni}.  Moreover,  $\Ref^f[\cdot]$ is valued in $\mathcal{S}^{2,rusc}$, where
 $ \mathcal{S}^{2,rusc}:= \{\phi\in \mathcal{S}^2: \phi \text{ is r.u.s.c.}\} $ (cf. Remark \ref{Rmk_the_jumps_of_C}).
 In the following proposition we give some  properties of the operator $\Ref^f$. Note that equalities (resp. inequalities) between processes are to be understood in the "up to indistinguishability"-sense. 
 
 We recall the notion of a  strong ${\cal E}^f$-supermartingale. 
\begin{Definition}\label{defmart}
Let $\phi$ be  a process in 
 $ \ \mathcal{S}^2$. Let $f$ be a Lipschitz driver.  The process $\phi$ is said to be a strong $\mathcal{E}^{^f}$-supermartingale (resp. a strong $\mathcal{E}^{^f}$-martingale) , if $\mathcal{E}^{^f}_{_{\sigma ,\tau}}(\phi_{\tau}) \leq \phi_{\sigma}$  a.s. (resp. $\mathcal{E}^{^f}_{_{\sigma ,\tau}}(\phi_{\tau}) = \phi_{\sigma}$  a.s.)\, on $\sigma \leq \tau$,  for all $ \sigma, \tau \in  \stopo$. 
\end{Definition}

Using the above comparison theorem and the ${\cal E}^f$-Mertens decomposition for strong (r.u.s.c.)  ${\cal E}^f$-supermartingales  in the case of a general filtration (cf. Theorem \ref{calmertens}), we show that the operator $\Ref^f$ satisfies the following properties.
\begin{Proposition}[Properties of the operator $\Ref^f$]\label{Properties_Ref}
Let $f$ be a Lipschitz driver satisfying Assumption \ref{Royer}.   
The operator $\Ref^f: S^{2}\rightarrow S^{2,rusc} $, defined in Definition \ref{definitionRef}, has the following properties:
\begin{enumerate}
\item
The operator $\Ref^f$ is nondecreasing,  
that is, for  $\xi, \xi'$ $\in S^{2}$ such that $\xi \leq  \xi'$ 
we have $\Ref^f[\xi]\leq \Ref^f[\xi']$.  
\item
If $\xi$ $\in S^2$ is a  (r.u.s.c.) strong ${\cal E}^f$-supermartingale, then $\Ref^f[\xi]=\xi.$ 
\item
For each $\xi$ $\in S^{2}$,  $\Ref^f[\xi]$ is a strong ${\cal E}^f$-supermartingale and satisfies $\Ref^f[\xi]\geq \xi$.
\end{enumerate}
\end{Proposition}
\begin{proof}
The first assertion follows from our comparison theorem for reflected BSDEs with irregular obstacles (Theorem \ref{thmcomprbsde}).\\
Let us prove the second assertion.
 Let $\xi$ be a (r.u.s.c.) strong ${\cal E}^f$-supermartingale in $ \mathcal{S}^2$. By definition of $\Ref^f$, we have to show that $\xi$ is the solution of the reflected  BSDE associated with driver $f$ and obstacle $\xi$. 
By the ${\cal E}^f$-Mertens decomposition for strong (r.u.s.c.)  ${\cal E}^f$-supermartingales in the case of a general filtration (Theorem \ref{calmertens}), together with Lemma \ref{theoreme representation}, there exists 
$(Z,k, h,A, C)\in \H^2 \times \H^2_{\nu}\times {\cal M}^{2,\bot} \times {\cal S}^2\times {\cal S}^2$ such that \\
$$-d  \xi_t =  f(t,\xi_{t}, Z_{t},k_t)dt -  Z_t  dW_t - \int_{{\bf E}}  k_t(e) \tilde{N}(dt,de)-dh_t +dA_t  +dC_{t-},\quad 0\leq t\leq T,$$
where  $A$ is predictable right-continuous nondecreasing with $A_0=0$, and  $C$ is adapted right-continuous  nondecreasing and purely discontinuous, with $C_{0-}=0$. Moreover, the Skorokhod conditions (for RBSDEs) are here trivially satisfied.  
Hence, $\xi=\Ref^f[\xi]$,  
which is the desired conclusion. \\
It remains to show the third assertion. By definition, the process $\Ref^f[\xi]$ is equal to $Y$, where $(Y,Z,k,h,A,C)$ is the solution  our reflected BSDE. Hence, $\Ref^f[\xi]=Y$ admits the decomposition \eqref{cmertens}, which, by Theorem \ref{calmertens}, implies that $\Ref^f[\xi]=Y$ is a strong ${\cal E}^f$-supermartingale. Moreover, by definition, $\Ref^f[\xi]=Y$ is greater than or equal to the obstacle $\xi$.
\end{proof}


With the help of the  above proposition, we show that the process $\Ref^f[\xi] $, that is, the first component of the solution of the RBSDE with (irregular) obstacle $\xi$, is characterized  in terms of the smallest strong $\mathcal{E}^f$-supermartingale greater than or equal to $\xi$.  
\begin{Theorem}[The operator $\Ref^f$ and the ${\cal E}^f$- Snell envelope operator] \label{Prop_characterization_Snell}
 Let $\xi$ be a   
process in ${\cal S}^2$
 and let $f$ be a  Lipschitz driver satisfying   Assumption \ref{Royer}. 
 The first component $Y=\Ref^f[\xi]$ of the solution to the reflected BSDE with parameters $(\xi,f)$  coincides with the ${\cal E}^f$-{\em Snell envelope} of $\xi$, 
that is, \emph{the smallest strong 
${\cal E}^f$-supermartingale greater than or equal to $\xi$.}
\end{Theorem}
\begin{proof}
By the third assertion of  Proposition \ref{Properties_Ref}, the process $Y=\Ref^f[\xi]$ is a strong 
${\cal E}^f$-supermartingale satisfying  $Y\geq \xi$. It remains to show the minimality property. 
Let $Y'$ be a strong ${\cal E}^f$-supermartingale such that  $Y'\geq \xi$. We have $\Ref^f[Y']\geq \Ref^f[\xi]$, due to the nondecreasingness of the operator $\Ref^f$ (cf. Proposition \ref{Properties_Ref}, 1st assertion). On the other hand, 
$\Ref^f[Y']=Y'$ (due to Proposition \ref{Properties_Ref}, 2nd assertion) and $\Ref^f[\xi]=Y$. 
 Hence, $Y'\geq Y$, which is the desired conclusion. 
%
\end{proof}
In the case of a right-continuous left-limited  obstacle $\xi$ the above  characterization has been established in  \cite{QuenSul2}; it has been generalized to the case of a right-upper-semicontinuous obstacle  in \cite[Prop. 4.4]{MG}. Let us note however that the arguments of the proofs given in \cite{QuenSul2} and in \cite{MG} cannot be adapted to our general framework. 

\section{Infinitesimal characterization of the value process in terms of an RBSDE in the completely irregular case}\label{sec_infinitesimal}
The following theorem is a direct consequence of Theorem \ref{Prop_characterization_Snell}  and Theorem  \ref{Prop_agrege}. It gives "an infinitesimal characterization" of the value process $(V_t)_{t\in[0,T]}$ of the non-linear problem \eqref{vvv}. 

\begin{Theorem}[Characterization in terms of an RBSDE]\label{caranonlinear} Let $(\xi_t)_{t\in[0,T]}$ be a 
process in ${\cal S}^2$
 and let $f$ be a  Lipschitz driver satisfying   Assumption \ref{Royer}. The value process $(V_t)_{t\in[0,T]}$ aggregating the  family $V=(V(S), \; S\in\stopo)$ defined by \eqref{vvv} coincides (up to indistinguishability) with the first component $(Y_t)_{t\in[0,T]}$  of the solution of our RBSDE with driver $f$ and obstacle $\xi$. In other words, we have, for all $S\in\stopo$, 
\begin{equation}\label{novel}
 Y_S=V_S  = {\rm ess} \sup_{\tau \in \T_{S,T}}{\cal E}^f_{S, \tau}(\xi_{\tau}) \text{ a.s. }
 \end{equation}
\end{Theorem}

By using this theorem, we derive the following corollary, which generalizes some results of classical optimal stopping theory (more precisely, the assertions (ii) and (iii) from Lemma 
\ref{pro}) to the case of an optimal stopping problem with (non-linear) $f$-expectation.

\begin{Remark} \label{difficulty} 
 Let us summarize our \emph{two-part} approach to the non-linear optimal stopping problem \eqref{vvv} in the case where  $\xi$ is \emph{completely irregular}:  First, we have applied \emph{a direct approach} to the   problem \eqref{vvv},  which consists in  showing that the value family $(V(S))_{S\in\stopo}$ can be aggregated by an optional process $(V_t)_{t\in[0,T]}$  and, then, in characterizing $(V_t)$ as the $\mathcal{E}^f$-Snell envelope of the (completely irregular) pay-off process $(\xi_t)$. On the other hand, we have applied \emph{an RBSDE-approach} which consists in establishing some results on RBSDEs with completely irregular obstacles (in particular, existence, uniqueness, and   a comparison result) and some useful properties of the operator $\Ref^f$, 
\footnote{
We emphasize that the proof of these properties (cf. Proposition \ref{Properties_Ref})  relies heavily on the ${\cal E}^f$-Mertens decomposition for strong ${\cal E}^f$-supermartingales 
(cf. Theorem \ref{calmertens}), which is obtained as a direct consequence of  the preliminary result (Theorem \ref{caracterisation}) established in the r.u.s.c. case.}
and then in using these properties to show that the unique  solution $(Y_t)$ of the RBSDE is equal to  the $\mathcal{E}^f$-Snell envelope of the completely irregular  obstacle. We have then deduced from those two parts (the direct part and the RBSDE-part) that $(Y_t)$ and $(V_t)$ coincide, which gives  an infinitesimal characterization for the value process $(V_t)$. 
%
 
 %

\end{Remark}

Finally, let us put together  some of the results for the non-linear optimal stopping problem \eqref{vvv}:  
\begin{compactenum}[i)]
\item 
$\bullet$ For any  reward process  $\xi\in\mathcal{S}^2$, we have the infinitesimal characterization\\ $V_t=Y_t=\Ref^f_t[\xi]$, for all $t$, a.s. (Theorem \ref{caranonlinear}).\\
 $\bullet$ Also,  $(V_t)_{t\in[0,T]}$ is the  ${\cal E}^f$-{\em Snell envelope} of the pay-off process $\xi$ (Theorem \ref{Prop_agrege}). 

\item If, moreover, $\xi$ is right-u.s.c.\,, then, for any $S\in\stopo$, for any $\varepsilon>0$, there exists an $L\varepsilon$- optimal stopping time for the problem at time $S$.
 (Theorem \ref{caracterisation}). 
\item If, moreover, $\xi$ is also left-u.s.c. along stopping times, then, for any $S\in\stopo$, there exists an  optimal stopping time for the problem at time $S$   
(Theorem \ref{Thm_existence_optimal}).
%
\end{compactenum}  
 

\section{Applications of Theorem \ref{caranonlinear}}\label{sec_app}
\subsection{Application to American options with a completely irregular payoff}\label{amerique}
In the following example, we set $E:= \R$, $\nu(de) := \lambda \delta_1(de)$, where $\lambda$ is a positive constant, and where $\delta_1$ denotes the Dirac measure at $1$. 
The process $N_t := N([0,t] \times \{1\})$ is then a Poisson process with parameter $\lambda$, and we have
$\tN_t := \tN([0,t] \times \{1\})  = N_t - \lambda t.$ 

We  assume that the filtration is the natural filtration associated with $W$ and $N$.

 We consider a financial market which consists of one risk-free asset, whose price process $S^0$ satisfies $dS_t^{0}=S_t^{0} r_tdt$, and two risky assets with price processes $S^{1},S^{2}$ satisfying:
 $$dS_t^{1}=S_{t^-}^{1}[\mu_t^1dt +  \sigma^1_t dW_t+ \beta_t^1d\tilde N_t];\quad dS_t^{2}=S_{t^-}^{2} [\mu^2_tdt+\sigma^2_tdW_t + \beta_t^2d\tilde N_t].$$
We suppose that the processes $\sigma^1,\sigma^2,$ $\beta^1,\beta^2,$ $r, \mu^1,\mu^2$ are 
predictable and bounded, with $\beta_t^i>-1$ for $i=1,2$. 
Let $\mu_t:= (\mu^1,\mu^2)'$ and let $\Sigma_t:= (\sigma_t, \beta_t)$ be the $2\times 2$-matrix with first column 
$\sigma_t:=(\sigma_t^1,\sigma_t^2)'$ and second column  $\beta_t:=(\beta_t^1,\beta_t^2)'$. 
We suppose that $\Sigma_t$ is invertible and that the coefficients of $\Sigma_t^{-1}$ are bounded.

We  consider an agent who can invest his/her initial wealth $x\in\R$ in the three assets. 

For $i=1,2$, we denote by $\varphi_t^i$ the amount invested in the $i^{\textit{th}}$ risky asset. 
A process $\varphi= (\varphi^1, \varphi^2)'$ belonging to ${\mathbb H}^2 \times  {\mathbb H}^2_{\nu}$ will be called a \emph{portfolio strategy}.

The value of the associated portfolio (or {\em wealth}) at time $t$ is denoted  by $X^{x, \varphi}_t$ (or simply by $X_t$). 
In the case of a perfect market, we have
\begin{align*}\label{portfolio}
dX_t & = (r_t X_t+\varphi_t^1 (\mu^1_t - r_t)+\varphi_t^2(\mu^2_t - r_t) ) dt +
(\varphi_t^1 \sigma^1_t + \varphi_t^2 \sigma^2_t) dW_t +
(\varphi_t^1 \beta^1_t + \varphi_t^2 \beta^2_t) d\tilde N_t \\
& = (r_t X_t+\varphi_t' (\mu_t - r_t{\bf 1}) ) dt+ \varphi_t' \sigma_t dW_t + 
\varphi_t' \beta_t  d\tilde N_t,
\end{align*}
where ${\bf 1}=(1,1)'$.
%
%
%
More generally, we will suppose that there may be some imperfections in the market,  taken into account via 
the {\em nonlinearity} of the
dynamics of the wealth and encoded in a Lipschitz driver $f$ satisfying Assumption~\ref{Royer} (cf. \cite{EQ96} or \cite{DQS4} for some examples). More precisely, 
we suppose that  the {\em wealth} process  $X^{x, \varphi}_t$ (also $X_t$)
satisfies  the forward differential equation: 
\begin{equation}\label{riche}
-dX_t= f(t,X_t, {\varphi_t}' \sigma_t, {\varphi_t}' \beta_t) dt - {\varphi_t}' \sigma_t dW_t-{\varphi_t}' \beta_t d\tilde N_t, \;  ; \;  X_0=x,
\end{equation}
 or,
equivalently, setting $Z_t= {\varphi_t}' { \sigma}_t$ and
  $k_t= {\varphi_t}' \beta_t $,
 \begin{equation}\label{wea}
-dX_t= f(t,X_t, Z_t,k_t ) dt -  Z_t dW_t- k_t d\tilde N_t ; \;  X_0=x.
\end{equation}
Note that $(Z_t, k_t)= {\varphi_t}' { \Sigma}_t$, which  is equivalent to ${\varphi_t}' = (Z_t, k_t)\, { \Sigma}_t^{-1}$.\\
This model includes the case of a perfect market, for which $f$ is a linear driver given by
$
f(t,y,z,k)= -r_t y- (z, k)\, { \Sigma}_t^{-1} (\mu_t - r_t{\bf 1}).
$

\begin{Remark}\label{richessemartingale}
Note that the wealth process $X^{x, \varphi}$ is an $\mathcal{E}^f$-martingale, since
$X^{x, \varphi}$ is the solution of the 
BSDE with driver $f$, terminal time $T$ and terminal condition $X_T^{x, \varphi}$. 
\end{Remark}

Let us consider an American option associated with terminal time $T$ 
and  payoff given by a process $(\xi_t)$ $\in {\cal S}^2$.
As is usual in the literature, the option's {\em superhedging price} at time $0$, denoted by $u_0$, is defined as the minimal initial wealth  enabling the seller  to invest in a portfolio whose value is greater than or equal to the payoff of the option at all times. 
More precisely, for each initial wealth $x$, we denote by ${\cal A} (x)$ the set of all portfolio strategies $ \varphi$ $\in$  ${\mathbb H}^2\times {\mathbb H}^2_{\nu}$ such that 
$X^{x, \varphi}_{ t} \geq \xi_t$,  for all $t \in [0,T]$ a.s. 
The {\em superhedging price}
of the American option is thus defined by
\begin{equation}\label{defsup}
u_0:= \inf \{x \in \R,\,\, \exists  \varphi \in {\cal A} (x) \}. \footnote{As shown in  assertion (iii) of Proposition \ref{americano}, the infimum in \eqref{defsup} is always attained. 
}
\end{equation}
Using the infinitesimal characterization of the value function \eqref{vvv} (cf. Theorem \ref{caranonlinear}), we show the following  characterizations of the superhedging price $u_0$, as well as the existence of a superhedging strategy.
\begin{Proposition}\label{americano} Let $(\xi_t)$ be an optional process such that $E[\esssup_{\tau\in\T_0} |\xi_\tau |^2] <  \infty. $  \\
(i) The {\em superhedging price} $u_0$ of the American option with payoff $(\xi_t)$ is equal to the value function $V(0)$ of our optimal stopping problem \eqref{eq_intro} at time $0$, that is
\begin{equation}\label{optimalstopping}
u_0=
\sup_{\tau \in {\stopo}}{\cal E}_{0,\tau}^f ( \xi_{\tau}).
\end{equation}
 (ii) We have $u_0=Y_0$, where $(Y,Z,k,h,A,C)$
 is the solution of the reflected BSDE \eqref{RBSDE} (with $h=0$).\\
 (iii) The  portfolio strategy $\hat{\varphi}$, defined by  $\hat{\varphi}_t{\, '} = (Z_t, k_t)\, { \Sigma}_t^{-1}$, is a superhedging strategy, 
that is, belongs to $ {\cal A} (u_0)$.
\end{Proposition}
%
In the case of a perfect market (for which $f$ is  linear) and a regular pay-off, the above result  reduces to a well-known result from the literature (cf. \cite{H}). 
Even in the case  of a perfect market, our result for a completely irregular pay-off is new.   

\begin{proof} The proof 
relies on Theorem \ref{caranonlinear} and similar arguments to those in  \cite{DQS4} (in the case of game options with RCLL payoffs and default). 
Note first that, by Theorem \ref{caranonlinear}, we have $\sup_{\tau \in \stopo} \mathcal{E}_{0, \tau }^f (\xi_\tau)= Y_0$. In order to prove the three first assertions of the above  theorem, 
it is thus sufficient to show that $u_0 =Y_0$ and $\hat{\varphi} \in {\cal A} (Y_0)$.

We first show that $\hat{\varphi}$ $\in$ $ {\cal A} 
(Y_0)$. By \eqref{wea}, the value $X^{Y_0, \hat{\varphi}}$ of the portfolio associated with initial wealth ${Y_0}$ and strategy 
$\hat{\varphi}$ satisfies:

$
dX_t^{Y_0, \hat{\varphi}}= - 
f(t,X_t^{Y_0, \hat{\varphi}},Z_t,k_t)dt + dM_t,
$ with initial condition $X_0^{Y_0, \hat{\varphi}}=Y_0$,
 where $ M_t := \int_0^t Z_s dW_s +\int_0^t k_s d\tilde N_s$. 
Moreover, since $Y$ is the solution of the reflected BSDE \eqref{RBSDE} (with $h=0$), we have
$
dY_t=- f(t,Y_t,Z_t,k_t)dt+dM_t -dA_t-dC_{t-}$.
%
Applying the comparison result for forward differential equations, we derive that
$X_t^{Y_0, \hat{\varphi}} \geq Y_t $, for all $t \in [0,T]$ a.s. Since 
$Y_t \geq \xi_t$, we thus get  $X_t^{Y_0, \hat{\varphi}} \geq 
\xi_t $ for all $t \in [0,T]$ a.s.\,It follows that $\hat{\varphi} \in {\cal A} 
(Y_0)$. 

We now show that $Y_0=u_0$. Since $\hat{\varphi} \in {\cal A} 
(Y_0)$, by definition of $u_0$ (cf. \eqref{defsup}), we derive that $Y_0 \geq u_0$. Let us now show that $u_0 \geq Y_0$.
%
%
%
Let $x \in \R$ be such that there exists a strategy $\varphi \in {\cal A} (x)$. We show that $x \geq Y_0$.
 Since $\varphi \in {\cal A} (x)$, we have
$X^{x, \varphi}_{t} \geq \xi_t$, for all $t \in [0,T]$ a.s. 
For each $\tau \in \mathcal{T}$ we thus get the inequality
$X^{x, \varphi}_{\tau } \geq \xi_\tau $ a.s.\,
By the non decreasing property 
of $\mathcal{E}^f$ together with the $\mathcal{E}^f$-martingale property of $X^{x, \varphi}$ (cf. Remark \ref{richessemartingale}),
we thus get
$
x =\mathcal{E}^f _{0,\tau }(X^{x, \varphi}_{\tau})\geq 
\mathcal{E}_{0, \tau }^f (\xi_\tau).
$
By taking the supremum over $\tau \in \stopo$, we derive that
$
x \geq 
\sup_{\tau \in \stopo} \mathcal{E}_{0, \tau }^f (\xi_\tau)=Y_0
$, where the equality holds by Theorem \ref{caranonlinear}. 
By definition of $u_0$ as an infimum (cf \eqref{defsup}), we get $u_0 \geq Y_0$, which, since $Y_0 \geq u_0$, yields that
$u_0 = Y_0$. 
%
\end{proof}
We now give some examples of American options with completely irregular pay-off. 
\begin{Example} 
We consider a pay-off process $(\xi_t)$ of the form $\xi_t:=h(S^1_t)$, for $t\in[0,T]$, where $h:\R\rightarrow \R$ is a (possibly irregular)   Borel function such that the process $(h(S_t))$ is optional and   $(h(S_t^1))\in\mathcal{S}^2$.    
In general, the pay-off $(\xi_t)$ is a completely irregular process. By the first two statements of  Proposition \ref{americano}, the superhedging price of the American option is equal to the value function of the optimal stopping problem \eqref{optimalstopping}, and is also
characterized as the solution of the reflected BSDE 
\eqref{RBSDE} with obstacle $\xi_t=h(S^1_t)$.

If $h$ is an uppersemicontinuous function on $\R$,  then the process  $(h(S_t^1))$ is optional, since an u.s.c. function can be written as the limit of a (non increasing) sequence of continuous functions. Moreover, 
 the process  $(h(S_t^1))$  is right-u.s.c.  and also left-u.s.c. along stopping times. The right-uppersemicontinuity of $(\xi_t)$ follows from the fact that the process $S^1$ is right-continuous; the left-uppersemicontinuity along stopping times of $(\xi_t)$ follows from the  fact that $S^1$  jumps only at totally inaccessible stopping times.           
In virtue of Proposition \ref{americano}, last statement, there exists  in this case  an {\em optimal exercise time} for the American option with payoff $\xi_t=h(S^1_t)$. 
A particular  example is given by the American digital call  option (with strike $K>0$), where $h(x):= {\bf 1}_{[K, + \infty[}(x)$. The function  $h$ is  u.s.c. on $\R$. The corresponding payoff process $\xi_t:={\bf 1}_{ S^1_t \geq K}$ is thus r.u.s.c and left-u.s.c. along stopping times in this case, which implies the existence of an optimal exercise time.


 In the case of the American  digital put option (with strike $K>0$), the corresponding payoff $\xi_t:={\bf 1}_{ S^1_t < K}$ is not r.u.s.c. 
We note that the pay-off of the American digital call and put options is in general neither left-limited nor right-limited. 

\end{Example}

\subsection{An application to RBSDEs}

The 
characterization (Theorem \ref{caranonlinear}) is also useful in the theory of RBSDEs in itself: it
 allows us to obtain  a priori estimates with universal constants for RBSDEs with completely irregular obstacles.
 
\begin{Proposition}[A priori estimates with universal constants] \label{oubli} 
Let $\xi$ and $ \xi'$ be two 
 processes in $ \mathcal{S}^2$. 
Let $f$ and $ f'$ be two Lipschitz drivers satisfying Assumption \ref{Royer} with common Lipschitz constant $K>0$. Let $(Y,Z,k)$ (resp. $(Y',Z',k')$)  be the three first components of the solution of the reflected BSDE associated with driver $f$ (resp. $f'$)  and obstacle $\xi$ (resp. $\xi'$). \\
Let $\overline{Y}:=Y-Y'$, $\overline{\xi}:=\xi-\xi'$, and
$\delta f_s:= f'(s,Y_s', Z_s', k_s')-f(s,Y_s', Z_s', k_s')$. \\
Let $\eta, \beta >0$ with $\beta\geq \dfrac{3}{\eta}+2K$ and $\eta \leq \dfrac{1}{K^2}$. 
For each $S \in  {\cal T}_{0,T}$, we have
\begin{equation}\label{eqA.1}
 \overline{Y_S}^2  \leq e^{\beta (T-S)}E[{\rm ess} \sup_{\tau \in \T_{S,T}} \overline{\xi_\tau}^2| \mathcal{F}_S]+ 
 \eta E[\int_S^Te^{\beta (s-S)}(\delta f_s)^2 ds|\mathcal{F}_S] \quad {\rm a.s.}\,
\end{equation}
\end{Proposition}
\begin{proof}
The proof is divided into two steps.\\
 {\bf Step 1}:
For each $\tau \in \stopo$, let $(X^{\tau}$, $\pi^{\tau}, l^{ \tau})$ (resp. $(X^{' \tau}$, $\pi^{' \tau}, l^{' \tau})$)  be the solution of the BSDE associated with driver $f$  (resp. $f'$), terminal time $\tau $  and terminal condition $\xi_\tau$  (resp. $\xi '_\tau$).
Set  $\overline{X}^{\tau}:=X^{\tau}-X^{' \tau}$. By an estimate on BSDEs (cf. Proposition $A.4$ in \cite{QuenSul}), we have 
\begin{equation*}
 (\overline{X}_S^{\tau})^2  \leq e^{\beta (T-S)} E[\overline{\xi}^2 \mid \mathcal{F}_S]+ \eta  E[\int_S^T e^{\beta (s-S)}{[(f-f')(s, X_s^{' \tau},\pi_s^{' \tau}, l_s^{' \tau}) ]^2ds} \mid \mathcal{F}_S]\quad {\rm a.s.}\,
\end{equation*}
from which we derive 
\begin{equation}\label{A.3}
 (\overline{X}_S^{\tau})^2 \leq e^{\beta (T-S)}E[{\rm ess} \sup_{\tau \in \T_{S,T}} \overline{\xi_\tau}^2| \mathcal{F}_S]+ 
 \eta E[\int_S^Te^{\beta (s-S)}(\overline{f}_s)^2 ds|\mathcal{F}_S] \quad {\rm a.s.},
\end{equation}
where $\overline{f}_s:= \sup_{y,z,k}|f(s,y,z,k)-f'(s,y,z,k)|$.
Now, by Theorem \ref{caranonlinear}, we have\\
 $Y_S ={\rm ess} \sup_{\tau \in \T_{S,T}} X_S^{ \tau}$ and $Y'_S ={\rm ess} \sup_{\tau \in \T_{S,T}} X_S^{' \tau}$.
 We thus  get
  $|\overline{Y}_S| \leq{\rm ess} \sup_{\tau \in \T_{S,T}}|\overline{X}_S^{\tau}|$. By \eqref{A.3},
 we  derive the inequality \eqref{eqA.1} with $\delta f_s$ replaced by $\overline{f}_s$.\\
{\bf Step 2}: 
Note that $(Y', Z', k')$ is the solution the RBSDE associated with obstacle $\xi '$ and 
driver $f(t,y,z,k) + \delta f_t$.  By applying the result of Step 1 to the driver $f(t,y,z,k)$ and the driver $f(t,y,z,k) + \delta f_t$ (instead of $f'$), we get the desired result.
\end{proof}
%

\section{Appendix} 
 Let $M,M' \in {\cal M}^2$. Recall  that $MM'- [ M, M'] $ is a martingale, and 
 that $\langle M, M' \rangle$ is defined as the {\em compensator} of the integrable finite variation process  $[ M, M']$. 
Using these properties we derive the following equivalent statements (cf., e.g.,\cite{Protter} IV.3 for details): \\
$\langle M, M' \rangle_t =0$,  $0 \leq t \leq T$ a.s. $\Leftrightarrow$  $[ M, M']_\cdot$ is a martingale $\Leftrightarrow$ $MM'$ is a martingale.
 \footnote{In this case, using he terminology of \cite{Protter} IV.3, the martingales $M$ and $M'$ are said to be {\em strongly orthogonal}.
 
Note also that, if $M,M' \in {\cal M}^2$,
using the terminology of \cite{Protter} IV.3, the martingales $M$ and $M'$ are said to be  {\em weakly orthogonal} if $(\, M\,, M'  \,)_{{\cal M}^2}=0$, that is $E[M_T M'_T]=0$.}

For the convenience of the reader, we state the following equivalences, which, to our knowledge, are not explicitly specified in the literature.
 \begin{Lemma} \label{orthogonality} 
For each $h \in {\cal M}^2$, the following properties are equivalent:
 \begin{description}
 \item[(i)]
For all predictable process $l$ $\in$ $\H_{\nu}^{2}$,  we have $\langle \, h\,, \int_0^\cdot l_s(e) \tilde N(ds de)\, \rangle_t =0$,  $0 \leq t \leq T$ a.s.
 \item[(ii)]
For all predictable process $l$ $\in$ $\H_{\nu}^{2}$, we have $(\, h \,, \int_0^\cdot \int_E l_s(e) \tilde N(ds de)  \,)_{{\cal M}^2}=0$. 
 \item[(iii)]
$M^P_{N}(\Delta h_\cdot\, \vert \tilde {\cal P})=0$, where $M^P_{N}(\, .\, \vert \tilde {\cal P})$ is the conditional expectation  given $\tilde {\cal P}:= \PC \otimes {\mathscr{E}}$ under the Doleans' measure $M^P_{N}$  associated to probability $P$ and random measure $N$.\footnote{For the definitions of $ M^P_{N}$ and  $M^P_{N}(\, .\, \vert \tilde {\cal P})$ see, e.g.\,, 
chapter III.1 (3.10) and (3.25) in \cite{J}.}
\end{description}
\end{Lemma}

\begin{proof} 
Let us show that {\bf (i)} $\Leftrightarrow$ {\bf (ii)}.
By definition of the scalar product $( \cdot , \cdot  )_{{\cal M}^2}$, we have 
$(\, h \,, \int_0^\cdot \int_E l_s(e) \tilde N(ds de)  \,)_{{\cal M}^2}$
 $=$ $E[\, \langle \, h\,,  \int_0^\cdot \int_E l_s(e) \tilde N(ds de)\,\rangle_T]. $
Hence, {\bf (i)} $\Rightarrow$ {\bf (ii)}. Let us show that {\bf (ii)} $\Rightarrow$ {\bf (i)}.
If for all  $l$ $\in$ $\H_{\nu}^{2}$, $E[ \,\langle \, h\,,  \int_0^\cdot l_s(e) \tilde N(ds de)\,\rangle_T]=0$, then, 
 for each bounded predictable process $\varphi$ $\in$ $\H^{2}$, we have
 $$E[\int_0^T \varphi_t \, d\langle \, h,  \int_0^\cdot \int_E  l_s(e) \tilde N(ds de)\,\rangle_t]= 
 E[ \, \langle \,h\,,  \int_0^\cdot \int_E \varphi_s l_s(e) \tilde N(ds de)\,\rangle_T]=0.$$
 since,  for each $M \in {\cal M}^2$, $ \varphi \cdot \langle h, M \rangle =  \langle h, \varphi. M \rangle$ (using the notation of \cite{DM2} or \cite{J}).
 By \cite{DM2} (Chap 6 II Th. 64 p141), this implies that the integrable-variation predictable process 
 $A_\cdot:=\langle \,h\,,  \int_0^\cdot l_s(e) \tilde N(ds de)\rangle_\cdot$ is equal to $0$, which gives that 
 {\bf (ii)} $\Rightarrow$ {\bf (i)}. Hence {\bf (i)} $\Leftrightarrow$ {\bf (ii)}.
 

It remains to show that {\bf (ii)} $\Leftrightarrow$ {\bf (iii)}. 
Note first that $(\, h \,, \int_0^\cdot \int_E l_s(e) \tilde N(ds de)  \,)_{{\cal M}^2}$ $=$ $E([ \,h\,,  \int_0^\cdot \int_E l_s(e) \tilde N(ds de)\,]_T)$
$=E(\int_{[0,T] \times E} \Delta h_s l_s(e)N(ds de))$ $=$ $M^P_{N}( \Delta h_\cdot \, l_\cdot )$.
Property {\bf (ii)} can thus be written as $M^P_{N}( \Delta h_\cdot \, l_\cdot )=0$ for all 
$l_\cdot$ $\in$ $\H_{\nu}^{2}$, which means that
$M^P_{N}(\Delta h_\cdot\, \vert \tilde {\cal P})=0$.
Hence, {\bf (ii)} $\Leftrightarrow$ {\bf (iii)}.
\end{proof}
\textbf{Proof of Lemma \ref{Lemma_estimate}:}
Let $\beta>0$ and  $\varepsilon>0$ be such that $\beta\geq\frac 1 {\varepsilon^2}$.
We note that $\tilde Y_T=\xi_T-\xi_T=0;$ moreover, we have
$-d\tilde Y_t=\tilde f(t) dt +d\tilde A_t+d\tilde C_{t-}-\tilde Z_t dW_t -\int_{E} \tilde k_t(e) \tilde N(dt,de) -d\tilde h_t.$

Thus we see that $\tilde Y$ is an {\em optional strong  semimartingale} in the vocabulary of \cite{Galchouk} 
%
 with decomposition 
$\tilde Y= \tilde Y_0+ M+A+B$, where
$M_t:=\int_0^t \tilde Z_s dW_s+\int_0^t \int_{E} \tilde k_s(e) \tilde N(ds,de) {\bf + \tilde h_t}$, $A_t:= -\int_0^t  \tilde f(s) ds- \tilde A_t$ and 
$B_t:=-\tilde C_{t-}$. 
Applying Gal'chouk-Lenglart's formula (more precisely Corollary A.2 in \cite{MG}) to $\e^{\beta t}\tilde Y_t^2$, and using that  $\tilde Y_T=0$, and the property  
$\langle h^c, W \rangle =0$, we get, almost surely, for all $t\in[0,T]$,
%
\begin{equation}\label{eq0_lemma_estimate}
\begin{aligned}
&\e^{\beta t}\tilde Y_t^2+ \int_{]t,T]} \e^{\beta s} \tilde Z_s^2 ds 
 + {\int_{]t,T]} \e^{\beta s} d\langle \tilde h^c \rangle _s }
 =  -\int_{]t,T]}\beta\e^{\beta s} (\tilde Y_{s})^2 ds+
2\int_{]t,T]} \e^{\beta s}\tilde Y_{s}\tilde f(s) ds\\&+2\int_{]t,T]} \e^{\beta s}\tilde Y_{s-}d\tilde A_s
-(\tilde M_T -\tilde M_t)
 -\sum_{t<s\leq T}\e^{\beta s}(\Delta \tilde Y_s)^2 +2\int_{[t,T[} \e^{\beta s}\tilde Y_{s}d\tilde C_{s}-\sum_{t\leq s<T}\e^{\beta s}(\tilde Y_{s+}-\tilde Y_{s})^2.
\end{aligned}
\end{equation}
where 
\begin{equation}\label{defM}
\tilde M_t:= 2\int_{]0,t]} \e^{\beta s}\tilde Y_{s-}\tilde Z_s d W_s+2
\int_{]0,t]} \e^{\beta s}\int_E \tilde Y_{s-}\tilde k_s(e) \tilde N(ds,de)
  +2\int_{]0,t]} \e^{\beta s}\tilde Y_{s-}d\tilde h_s.
 \end{equation}
By the same arguments as in \cite{MG} (cf. the proof of Lemma 3.2 in \cite{MG} for details), since $\beta\geq\frac 1 {\varepsilon^2}$, we obtain the following estimate for the sum of the first and the  second term  on the r.h.s. of equality  \eqref{eq0_lemma_estimate}:
%
$
-\int_{]t,T]}\beta\e^{\beta s} (\tilde Y_{s})^2 ds+2\int_{]t,T]}  \e^{\beta s}\tilde Y_{s}\tilde f(s) ds \leq \varepsilon^2 \int_{]t,T]}  \e^{\beta s}\tilde f^2(s) ds.
$\\
We also have that  $\int_{[t,T[} \e^{\beta s}\tilde Y_{s}d\tilde C_{s}\leq 0$ and $\int_{]t,T]}$ $ \e^{\beta s}\tilde Y_{s-}d\tilde A_s \leq 0.$
We give the detailed arguments for the second inequality (the arguments for the first are similar). We have $\int_{]t,T]}$ $ \e^{\beta s}\tilde Y_{s-}d\tilde A_s=\int_{]t,T]}$ $ \e^{\beta s}\tilde Y_{s-}d A^1_s-\int_{]t,T]}$ $ \e^{\beta s}\tilde Y_{s-}d A^2_s$. For the first term, we write
$\int_{]t,T]}$ $ \e^{\beta s}\tilde Y_{s-}d A^1_s=\int_{]t,T]}$ $ \e^{\beta s} (Y^1_{s-}-Y^2_{s-})d A^1_s=\int_{]t,T]}$ $ \e^{\beta s} (Y^1_{s-}-\overline{\xi}_s)d A^1_s+\int_{]t,T]}$ $ \e^{\beta s}(\overline{\xi}_s- Y^2_{s-})d A^1_s$. The second summand is nonpositive as $Y^2_{s-}\geq \overline{\xi}_s$ (which is due to $Y^2_{s}\geq \xi_s,$ for all $s$). The first summand is equal to $0$ due to the Skorokhod condition for $A^1$.  Hence, $\int_{]t,T]}$ $ \e^{\beta s}\tilde Y_{s-}d A^1_s\leq 0$.  By similar arguments, we see that $-\int_{]t,T]}$ $ \e^{\beta s}\tilde Y_{s-}d A^2_s\leq 0$. Hence, $\int_{]t,T]}$ $ \e^{\beta s}\tilde Y_{s-}d\tilde A_s \leq 0.$  

The above observations, together with equation \eqref{eq0_lemma_estimate}, yield that a.s., for all $t\in[0,T]$,  
\begin{equation}\label{eq7_lemma_estimate}
\begin{aligned}
\e^{\beta t}\tilde Y_t^2+\int_{]t,T]} \e^{\beta s} \tilde Z_s^2 ds  + \int_{]t,T]} \e^{\beta s} d\langle \tilde h^c \rangle _s &\leq \varepsilon^2 \int_{]t,T]}  \e^{\beta s}\tilde f^2(s) ds-
(\tilde M_T -\tilde M_t)-\sum_{t<s\leq T}\e^{\beta s}(\Delta \tilde Y_s)^2,
\end{aligned}
\end{equation}
from which we derive  estimates for $\|\tilde Z\|^2_\beta$,
$\|\tilde k\|^2_{\nu,\beta}$, $\|\tilde h\|^2_{\beta, \mathcal{M}^2}$, and then an estimate for $\vvertiii{\tilde Y}^2_\beta.$
\paragraph*{ Estimate for $\|\tilde Z\|^2_\beta$, 
$\|\tilde k\|^2_{\nu,\beta}$ and $\|\tilde h \|^2_{\beta, {\cal M}^2}$}
Note first that we have: 
\begin{equation*}
\begin{aligned}
  \sum_{t< s\leq T}\e^{\beta s}(\Delta \tilde h_s) ^2&+\int_{]t,T]} \e^{\beta s} ||\tilde k_s||_{\nu}^2 ds-\sum_{t<s\leq T}\e^{\beta s}(\Delta \tilde Y_s)^2
=-\sum_{t<s\leq T}\e^{\beta s}(\Delta \tilde A_s) ^2\\&-\int_{]t,T]} \e^{\beta s}\int_E \tilde k_s^2(e)  \tilde N(ds,de)
  -2 \sum_{t<s\leq T}\e^{\beta s} \Delta \tilde A_s \Delta \tilde h_s -2 \sum_{t<s\leq T}\e^{\beta s}\tilde k_s(p_s) \Delta \tilde h_s ,
\end{aligned}
\end{equation*}
where, we have used the fact that the processes $A_\cdot$ and $N(\cdot,de)$ "do not have jumps in common",  since $A$ (resp. $N(\cdot,de)$) jumps only at predictable (resp. totally inaccessible) stopping times.\\ 
By adding the term $\int_{]t,T]} \e^{\beta s}|| \tilde k_s||_{\nu}^2 ds$ $+   \sum_{t< s\leq T}\e^{\beta s}(\Delta \tilde h_s) ^2$ on both sides of inequality \eqref{eq7_lemma_estimate}, by using the above computation   and the well-known equality $[\tilde h]_t= \langle \tilde h^c \rangle_t+ \sum (\Delta \tilde h)_s^2$, we get 
\begin{equation}\label{eq_last_000}
\begin{aligned}
 \e^{\beta t}\tilde Y_t^2+\int_{]t,T]} \e^{\beta s} \tilde Z_s^2 ds &+\int_{]t,T]} \e^{\beta s}|| \tilde k_s||_{\nu}^2 ds  +  \int_{]t,T]} \e^{\beta s} d[ \tilde h] _s 
\leq \varepsilon^2 \int_{]t,T]}  \e^{\beta s}\tilde f^2(s) ds - (M'_T- M'_t)\\
& -2 \sum_{t<s\leq T}\e^{\beta s} \Delta \tilde A_s \Delta \tilde h_s -2 \int_t^T d[\tilde h\,, \int_0^\cdot \int_E \e^{\beta s}\tilde k_s(e) \tilde N(ds,de)\,]_s,
\end{aligned}
\end{equation}
with $M'_t= \tilde M_t + \int_{]t,T]} \e^{\beta s}\int_E  \tilde k_s^2(e)  \tilde N(ds,de)$ (where $\tilde M$ is given by \eqref{defM}).\\
  By classical arguments, which  use Burkholder-Davis-Gundy inequalities, we can show that the local martingale $M'$ is a martingale.
Moreover, since $\tilde h$  $\in$ 
 ${\cal M}^{2,\bot}$, by Remark \ref{2i}, 
  we derive that  the expectation of the last term of the above inequality \eqref{eq_last_000} is equal to $0$.
Furthermore, since $\tilde h$ is a martingale, for each predictable stopping time $\tau$, we have 
$E[\Delta \tilde h_\tau / {\cal F}_{\tau-}]=0$ (cf., e.g., Chapter I, Lemma (1.21) in \cite{J}). Moreover, 
since $\tilde A$ is predictable, $\Delta \tilde A_\tau$ is ${\cal F}_{\tau-}$-measurable (cf., e.g., Chap I (1.40)-(1.42) in \cite{J}), which implies that
$E[\Delta \tilde A_\tau \Delta \tilde h_\tau / {\cal F}_{\tau-}]=
\Delta \tilde A_\tau E[ \Delta \tilde h_\tau / {\cal F}_{\tau-}]=0$. We thus get 
$E[\sum_{0< s \leq T}\e^{\beta s}\Delta \tilde A_s \Delta \tilde h_s]=0.$

By applying \eqref{eq_last_000} with $t=0$, and by taking expectations on both sides of the resulting inequality, we obtain
$\tilde Y_0^2+ \|\tilde Z\|^2_\beta+\|\tilde k\|^2_{\nu,\beta}+  \|\tilde h \|^2_{\beta, {\cal M}^2} \leq \varepsilon^2\|\tilde f\|^2_\beta.$        
We deduce that $\|\tilde Z\|^2_\beta\leq \varepsilon^2\|\tilde f\|^2_\beta$, $ \|\tilde k\|^2_{\nu,\beta}\leq \varepsilon^2\|\tilde f\|^2_\beta  $ and $\|\tilde h \|^2_{\beta, {\cal M}^2} \leq \varepsilon^2\|\tilde f\|^2_\beta$, which are the desired estimates \eqref{eq_initial_Lemma_estimate}.


\paragraph*{Estimate for $\vvertiii{\tilde Y}^2_\beta$}
%
From inequality \eqref{eq7_lemma_estimate} we derive that,
for all $\tau \in \stopo$,  a.s.,\\
$\e^{\beta \tau}\tilde Y_\tau^2 \leq \varepsilon^2 \int_{]\tau,T]}  \e^{\beta s}\tilde f^2(s) ds-
(\tilde M_T -\tilde M_\tau),$ 
where $\tilde M$ is  given by \eqref{defM}.

%
Using first Chasles' relation for  stochastic integrals, then taking the essential supremum over $\tau\in\stopo$ and  the expectation on both sides of the above inequality, we obtain   
\begin{equation}\label{eq9_lemma_estimate} 
\begin{aligned}
E[\esssup_{\tau \in \stopo}\e^{\beta \tau}\tilde Y_\tau^2] \leq \varepsilon^2 \|\tilde f\|^2_\beta  
&+2E[\esssup_{\tau \in \stopo}|\int_0^\tau \e^{\beta s}\tilde Y_{s-}\tilde Z_s d W_s|] +  2E[\esssup_{\tau \in \stopo}|\int_0^\tau \e^{\beta s}\tilde Y_{s-}d\tilde h_s |]  \\
& + 2E[\esssup_{\tau \in \stopo}|\int_{]0,\tau]} \e^{\beta s}\int_E \tilde Y_{s-}\tilde k_s(e) \tilde N(ds,de)|].
\end{aligned}
\end{equation}
Let us  consider the third term of the r.h.s. of the inequality \eqref{eq9_lemma_estimate}.
By Burkholder-Davis-Gundy inequalities, we have
$
 E[\esssup_{\tau \in \stopo}|\int_0^\tau \e^{\beta s}\tilde Y_{s-}d\tilde h_s |] 
\leq c E[ \sqrt{ \int_0^T \e^{2\beta s} \tilde Y_{s-}^2 d[\tilde h]_s} ]. 
$
 This inequality and the trivial inequality $ab\leq \frac 1 2 a^2+\frac 1 2 b^2$ lead to  
$$
2E[\esssup_{\tau \in \stopo}|\int_0^\tau \e^{\beta s}\tilde Y_{s-}d\tilde h_s |] 
\leq  E\left[ \sqrt{\frac 1 2\esssup_{\tau\in\stopo}\e^{\beta \tau}\tilde Y_{\tau}^2}\sqrt{8c^2
\int_0^T \e^{\beta s}d[\tilde h]_s }\right]
\leq \frac 1 4 \vvertiii{\tilde Y}^2_\beta + 
 4c^2\|\tilde h\|^2_{\beta, \mathcal{M}^2}. 
$$
 By using similar arguments, we get 
$
2E[\esssup_{\tau \in \stopo}\int_0^\tau \e^{\beta s}\tilde Y_{s-}\tilde Z_s d W_s]\leq 
\frac 1 4 \vvertiii{\tilde Y}^2_\beta+ 4c^2\|\tilde Z\|^2_{\beta},
$
and a similar estimate for the last term in \eqref{eq9_lemma_estimate}.
By \eqref{eq9_lemma_estimate}, we thus have
$\frac 1 4 \vvertiii{\tilde Y}^2_\beta \leq \varepsilon^2 \|\tilde f\|^2_\beta+4c^2 ( \|\tilde Z\|^2_\beta+ \|\tilde k\|^2_{\nu,\beta} +  \|\tilde h\|^2_{\beta, {\cal M}^2}).$ 
Using the estimates for $\|\tilde Z\|^2_\beta$, 
$\|\tilde k\|^2_{\nu,\beta}$ and $\|\tilde h \|^2_{\beta, {\cal M}^2}$(cf. \eqref{eq_initial_Lemma_estimate}), we thus get
$ \vvertiii{\tilde Y}^2_\beta \leq 4\varepsilon^2(1+12c^2) \|\tilde f\|^2_\beta,$
which is the desired result.  
{\hspace*{.1pt}\hspace*{\fill}\BOX\vskip\baselineskip} 
\begin{Remark}\label{EDSRclassique}
 We note that this proof shows that  the estimates \eqref{eq_initial_Lemma_estimate} and
\eqref{deuxeq} also hold in the simpler case of a non reflected BSDE. From this result, together with Lemma \ref{theoreme representation}, and using the same arguments as in the proof of Theorem \ref{rexiuni}, we easily derive the existence and the uniqueness of the solution of the non reflected BSDE with general filtration from Definition \ref{BSDE}. Similarly, we can show the comparison result for non reflected BSDEs with general filtration under the Assumption \ref{Royer}. 
\end{Remark}


\begin{Lemma}\label{compref} 
Let $f$ be a  Lipschitz driver satisfying Assumption  \ref{Royer}. Let $A$  be a nondecreasing 
 right-continuous predictable process in ${\cal S}^2$ with $A_0=0$
 and let $C$ be a nondecreasing 
  right-continuous adapted purely discontinuous process in ${\cal S}^2$ with $C_{0-}=0$.\\
   Let $(Y,Z,k,h) \in {\cal S}^2 \times \mathbb{H}^2\times \mathbb{H}^2_\nu\times {\cal M}^{2, \bot}$  satisfy \\
$ -d  Y_t  \displaystyle =  f(t,Y_{t}, Z_{t}, k_t)dt + dA_t + d C_{t-}-  Z_t dW_t-\int_E k_t(e)\tilde{N}(dt,de) - dh_t, \quad 0\leq t\leq T.
$\\
%
Then the process $(Y_t)$ is a strong ${\cal E}^f$-supermartingale. 
\end{Lemma}

The proof is omitted since it relies on the same arguments as those used in the proof of the same result shown in \cite{MG} in the particular case when the filtration is associated with $W$ and $N$ (cf. Proposition A.5 in \cite{MG}), as well as on some specific arguments, due to the general filtration, which are similar to those used in the proof of the previous lemma. 
\section{Complements: The strict value}

In this section we give some complements on a closely related (non-linear) optimal stopping problem.\\ 
Let $S$ be a stopping time in $\stopo$. 
We denote by $\mathcal{T}_{S^+}$ the set of stopping times
$\tau\in \stopo$ with $\tau>S $ a.s. on $\{S<T\}$ and $\tau=T$ a.s. on $\{S=T\}$.
\emph{The strict value} $V^+(S)$ (at time $S$) of the non-linear optimal stopping problem is defined by
\begin{align}\label{nuValeur}
V^{+}(S):= \esssup_{\tau \in \mathcal{T}_{S^+}} {\cal E}^{f}_{S,\tau} (\xi_\tau).
\end{align}
We note that $V^{+}(S)= \xi_T$ a.s. on $\{S=T\}$.\\
Using the same arguments as for the value family $(V(S))_{S\in\stopo}$, we show that
\begin{Proposition}\label{Property_0}
The strict value family  $(V^+(S))_{S\in\stopo}$ is a strong $\mathcal{E}^f$-supermartingale family. There exists a unique right-uppersemicontinuous optional process, denoted by $(V^+_t)_{t\in[0,T]}$, which aggregates the family $(V^+(S))_{S\in\stopo}$. The process $(V^+_t)_{t\in[0,T]}$ is a strong $\mathcal{E}^f$-supermartingale.  
\end{Proposition}

The following theorem connects the above  strict value process $(V^+_t)_{t\in[0,T]}$  with the process of right-limits $(V_{t+})_{t\in[0,T]}$, where $(V_t)$ denotes as before the value process of our non-linear problem \eqref{vvv}.

\begin{Theorem}\label{thm_value_nu}
\begin{description}
\item[(i)] The strict value process $(V^{+}_t)$ is right-continuous. 
\item[(ii)] For all $S\in\stopo$, $V^{+}_S=V_{S+}$ a.s.   
\item[(iii)] For all $S\in\stopo$, $V_S=V^{+}_S\vee \xi_S$ a.s. 
\end{description}
\end{Theorem}
 
The proof of the theorem uses the following preliminary result which states that the strict value process $(V^{+}_t)$ is right-continuous along stopping times in ${\cal E}^{f}$-conditional expectation. 

\begin{Lemma}[Right-continuity along stopping times in $\mathcal{E}^f$-conditional expectation]\label{Property_2}
The strict value process $(V^{+}_t)$ is right-continuous along stopping times in ${\cal E}^{f}$-expectation, in the sense that for each $\theta \in \stopo$, and for each sequence of stopping times $(\theta_n)_{n\in\mathbb{N}}$  belonging to $\stopo$  such that $\theta_n\downarrow \theta$, we have 
\begin{equation}\label{RCEv+}
\lim_{n\to \infty} \uparrow {\cal E}^f_{\theta, \theta_n} (V^{+}_{\theta_n}) = V^{+}_{\theta}\quad {\rm a.s.}
\end{equation}
\end{Lemma}

For the proof, we recall the following classical statement:
\begin{Remark}\label{classique} Let $(\Omega, {\cal F}, P)$ be a probability space. 
 Let $A \in  {\cal F}$. Let $(X_n)$ be a sequence of real valued random variables. Suppose that $(X_n)$ converges a.s. on $A$ to a random variable $X$. 
 Then, for each $\varepsilon >0$, 
 $\lim_{n \to + \infty}P(  \{ \vert X- X_n\vert < \varepsilon \} \cap A )= P(A).$\\
From this property, it follows that for each $\varepsilon >0$,   there exists $n_0 \in\mathbb{N}$ such that for all $n \geq n_0$, 
$P(  \{ \vert X- X_n\vert < \varepsilon \} \cap A ) \geq \frac{P(A)}{2}.$
\end{Remark}
\textbf{Proof of Lemma \ref{Property_2}:}
 Let $n\in\mathbb{N}$. By the consistency property of ${\cal E}^{f}$, we have 
\begin{equation}\label{cons}
{\cal E}^f_{\theta, \theta_n} (V^+_{\theta_n})= {\cal E}^f_{\theta, \theta_{n+1}}
\left( {\cal E}^f_{\theta_{n+1}, \theta_n}
 (V^+_{\theta_n})  \right)\quad {\rm a.s.}
 \end{equation}
Now, since the process $(V^+_t)$ is a strong ${\cal E}^f$- supermartingale, we have 
$ {\cal E}^f_{\theta_{n+1}, \theta_n} (V^+_{\theta_n}) \leq V^+_{\theta_{n+1}}$ a.s.  Using this inequality, together with  equality \eqref{cons}  and 
 the monotonicity of ${\cal E}^f_{\theta, \theta_{n+1}}$, we obtain
$${\cal E}^f_{\theta, \theta_n} (V^+_{\theta_n})\leq  {\cal E}^f_{\theta, \theta_{n+1}}
(V^+_{\theta_{n+1}}) \quad {\rm a.s.}$$
Since this inequality holds for each $n\in\mathbb{N}$, we derive that the sequence of random variables $\left({\cal E}^f_{\theta, \theta_n} (V^+_{\theta_n})\right)_{n\in\mathbb{N}}$ is nondecreasing.  
Moreover, since the process $(V^+_t)$ is a strong ${\cal E}^f$- supermartingale, we have $ {\cal E}^f_{\theta, \theta_n} (V^+_{\theta_n}) \leq V^+_\theta$ a.s. 
for each $n\in\mathbb{N}$. By taking the limit as $n$ tend to $+ \infty$, we thus get
\begin{equation*}
\lim_{n\to \infty} \uparrow {\cal E}^f_{\theta, \theta_n} (V^+_{\theta_n}) \leq V^+_\theta\quad {\rm a.s.}
\end{equation*}
It remains to show the converse inequality:
\begin{equation}\label{autreinegalite}
\lim_{n\to \infty}\uparrow  {\cal E}^f_{\theta, \theta_n} (V^+_{\theta_n}) \geq V^+_\theta\quad {\rm a.s.}
\end{equation}
Suppose, by way of contradiction, that this inequality does not hold. 
Then, there exists a constant $\alpha>0$ such that the event $A$ defined by 
$$A:= \{ \lim_{n\to \infty}\uparrow  {\cal E}^f_{\theta, \theta_n} (V^+_{\theta_n}) \leq V^+_\theta - \alpha \}$$
satisfies $P(A) >0$. By definition of $A$, we have 
\begin{equation}\label{rose}
\lim_{n\to \infty}\uparrow  {\cal E}^f_{\theta, \theta_n} (V^+_{\theta_n}) +  \alpha \leq V^+_\theta \quad {\rm a.s.}\,\, {\rm on}\,\, A.
\end{equation}
As for the value function, there exists an optimizing sequence $(\tau_p)_{p\in\mathbb{N}}$ 
 for the strict value function $V^+_\theta$, that is, such that, for each $p\in\mathbb{N}$,  $\tau_p \in 
 {\cal T}_{\theta^+}$, and such that 
 $$V^+_\theta= \lim_{p\to \infty}\uparrow {\cal E}^f_{\theta, \tau_p} (\xi_{\tau_p  })
 \quad {\rm a.s.}$$
 By Remark \ref{classique} (applied with $\varepsilon = \frac{\alpha}{2}$), we derive that there exists $p_0 \in\mathbb{N}$ such that the event $B$ defined by 
 $$B:= \{ V^+_\theta \leq  {\cal E}^f_{\theta, \tau_{p_0}} (\xi_{\tau_{p_0}  }) +\frac{\alpha}{2} \} \cap A$$
satisfies $P(B) \geq \frac{P(A)}{2}$. Denoting $\tau_{p_0}$ by $\theta'$, we have 
$$V^+_\theta \leq  {\cal E}^f_{\theta, \theta'} (\xi_{\theta'  }) +\frac{\alpha}{2}
\quad {\rm a.s.}\,\, {\rm on}\,\, B.$$

By the inequality \eqref{rose}, we derive that
\begin{equation}\label{rose2}
\lim_{n\to \infty}\uparrow  {\cal E}^f_{\theta, \theta_n} (V^+_{\theta_n}) +\frac{\alpha}{2} \leq  {\cal E}^f_{\theta, \theta'} (\xi_{\theta'  }) \quad {\rm a.s.}\,\, {\rm on}\,\, B.
\end{equation}

Let us first consider the simpler case where $\theta<T$ a.s.\\  In this case,  since $\theta'\in \mathcal{T}_{\theta^+}$, we have $\theta'>\theta$ a.s.
Hence, we have $\Omega= \displaystyle{\cup_{n\in \mathbb{N}}} \uparrow \{\theta'>\theta_n\}$ a.s.\\
Define the stopping time  $\overline \theta_n := \theta'{\bf 1}_{\{\theta'>\theta_{n}\}}+T {\bf 1}_{\{\theta'\leq \theta_{n}\}}. $
We  note that $\overline \theta_n\in\T_{\theta_{n}^{\;\:+}}$   for each $n$ $\in\mathbb{N}$. Moreover, $\lim_{n\to \infty} \overline \theta_n = \theta'$ a.s. and 
 $\lim_{n\to \infty}\xi_{ \overline \theta_n} = \xi_{\theta'}$ a.s.
 By the continuity property of $ {\cal E}^f$ with respect to terminal condition and 
 terminal time, we get
 $$\lim_{n\to \infty} {\cal E}^f_{\theta, \overline \theta_n}(\xi_{ \overline \theta_n})
 =  {\cal E}^f_{\theta, \theta'} (\xi_{\theta'  }) \quad {\rm a.s.}$$
 By Remark \ref{classique}, we derive that there exists 
  $n_0 \in\mathbb{N}$ such that the event $C$ defined by 
 $$C:= \{ \vert {\cal E}^f_{\theta, \theta'} (\xi_{\theta'  })  -  {\cal E}^f_{\theta,  \overline \theta_{n_0}} (\xi_{\overline \theta_{n_0}  }) \vert \leq \frac{\alpha}{4} \} \cap B$$
satisfies $P(C)>0$. By the inequality \eqref{rose2}, we derive that 
\begin{equation}\label{eqC}
\lim_{n\to \infty}\uparrow  {\cal E}^f_{\theta, \theta_n} (V^+_{\theta_n}) +\frac{\alpha}{4} \leq   {\cal E}^f_{\theta,  \overline \theta_{n_0}} (\xi_{\overline \theta_{n_0}  }) \quad {\rm a.s.}\,\, {\rm on}\,\, C.
\end{equation}
Now, by the consistency of $ {\cal E}^f$, we have
$$ {\cal E}^f_{\theta,  \overline \theta_{n_0}} (\xi_{\overline \theta_{n_0}  }) = 
{\cal E}^f_{\theta,   \theta_{n_0}}\left( {\cal E}^f_{\theta_{n_0},  \overline \theta_{n_0}}
 (\xi_{\overline \theta_{n_0}  })  \right)
 \leq {\cal E}^f_{\theta,   \theta_{n_0}} ( V^+_{\theta_{n_0}} ) \quad {\rm a.s.,}$$
where the last inequality follows from the fact that $\overline \theta_{n_0}\in \T_{\theta_{n_0}^{\;\;\;+}}$  and from the definition of $V^+_{\theta_{n_0}}$.
By \eqref{eqC}, we thus derive that 
$$\lim_{n\to \infty}\uparrow  {\cal E}^f_{\theta, \theta_n} (V^+_{\theta_n}) +\frac{\alpha}{4} \leq   {\cal E}^f_{\theta,   \theta_{n_0}} ( V^+_{\theta_{n_0}} ) \quad {\rm a.s.}\,\, {\rm on}\,\, C,$$
 which gives a contradiction. Hence, the desired inequality \eqref{autreinegalite}  holds. 

Let us now consider a general  $\theta\in \stopo$.\\ 
On the set $\{\theta = T\}$,  we have $\theta_n =\theta$ a.s. for all $n$. Hence,
 on $\{\theta = T\}$, we have 
$\lim_{n\to \infty}{\cal E}^f_{\theta, \theta_n} (V^+_{\theta_n}) = V^+_\theta\; {\rm a.s.}$ 
On the set $\{\theta < T\}$, using  the same arguments as above with 
$\overline \theta_n = \theta'{\bf 1}_{\{\theta'>\theta_{n}\}\cap \{T>\theta\}}+T {\bf 1}_{\{\theta'\leq \theta_{n}\}
 \cup \{T=\theta\}}$, we show the inequality \eqref{autreinegalite}.
 The proof is thus complete.
 \fproof
 
We are now ready to prove the theorem.\\
\textbf{Proof of Theorem \ref{thm_value_nu}:}
The proof of (i) is based on the previous Lemma \ref{Property_2} and on a result from the general theory of processes. Let $S\in\stopo$ and let  $(S_n)$ be a non-increasing sequence of stopping times in $\mathcal{T}_{S +}$ with $\lim \downarrow S_n=S$ a.s. By applying Lemma \ref{Property_2} and the continuity property of $\mathcal{E}^f$-expectations 
with respect to the terminal condition and to the terminal time, 
 we get 
$$V^{+}_S=\lim_{n\to\infty} \mathcal{E}^f_{S,S_n}(V^{+}_{S_n})=\mathcal{E}^f_{S,S}(\lim_{n\to\infty} V^{+}_{S_n})=\lim_{n\to\infty} V^{+}_{S_n},$$
where we have used that $\lim_{n\to\infty} V^{+}_{S_n}$ exists, as $(V^{+}_t)$ is a strong $\mathcal{E}^f$-supermartingale, and hence has right limits.    
The above equality shows that the process $(V^{+}_t)$ is right-continuous along stopping times. By  Proposition 2 in \cite{DelLen2}, we conclude that  $(V^{+}_t)$ is right-continuous.     \\
We now show (ii).   
Let $S\in\stopo$. Let $(S_n)$ be a  non-increasing sequence of stopping times in $\mathcal{T}_{S^+}$ with $\lim \downarrow S_n=S$ a.s. We know that $V_\tau\geq V^{+}_\tau$ a.s., for all $\tau\in\stopo$.
Hence, $V_{S_n}\geq V^{+}_{S_n}$ a.s., for all $n$. We derive that $\lim_{n\to\infty} V_{S_n} \geq \lim_{n\to\infty} V^{+}_{S_n}$ a.s.  Using this and the right-continuity  of $V^{+}$ established in (i), gives
 $V_{S+}\geq V^{+}_S$ a.s. In order to show the converse inequality, we first show  
 \begin{equation}\label{eq_0_thm_strict_nu}
 \mathcal{E}^f_{S,S^n}(V_{S_n})\leq V_S^{+} \text{ a.s. for all } n.  
 \end{equation}  
 We fix $n$ and we take $(\tau^p)\in\mathcal{T}_{S_n}$  an optimizing sequence for the problem with value $V_{S_n}$, i.e.
 $V_{S_n}=\lim_{p\to\infty}  \mathcal{E}^{f}_{S_n,\tau_p}(\xi_{\tau_p}).$   We have 
 \begin{equation}\label{eq1_thm_strict_nu}
 \mathcal{E}^f_{S,S_n}(V_{S_n})=\mathcal{E}^f_{S,S_n}(\lim_{p\to\infty}  \mathcal{E}^{f}_{S_n,\tau_p}(\xi_{\tau_p}))=\lim_{p\to\infty} \mathcal{E}^f_{S,S_n}(  \mathcal{E}^{f}_{S_n,\tau_p}(\xi_{\tau_p})) \text{ a.s.},
 \end{equation} where we have used    the continuity property of $\mathcal{E}^f_{S,S^n}(\cdot)$ with respect to the terminal condition (recall that here $n$ is fixed). 
 Using the consistency property of $\mathcal{E}^f$-expectations, we get $\mathcal{E}^f_{S,S_n}(  \mathcal{E}^{f}_{S_n,\tau_p}(\xi_{\tau_p}))= \mathcal{E}^{f}_{S,\tau_p}(\xi_{\tau_p})\leq V^{+}_S$ a.s. (where for the  inequality we have used that $\tau_p\in\mathcal{T}_{S^+}$). From this, together with equation \eqref{eq1_thm_strict_nu}, we derive the desired inequality \eqref{eq_0_thm_strict_nu}. 
From inequality \eqref{eq_0_thm_strict_nu}, together with the continuity of $\mathcal{E}^f$-expectations with respect to the terminal time and the terminal condition, we derive 
$ V^{+}_S\geq \lim_{n\to\infty}\mathcal{E}^f_{S,S^n}(V_{S_n})=\mathcal{E}^f_{S,S}(V_{S+})=V_{S+}  $ a.s. Hence, $V^{+}_S\geq V_{S+}$
 a.s., which, together with the previously shown converse inequality, proves the equality $V_{S+} = V^{+}_S$ a.s.\\ 
Statement (iii) is a direct consequence of part (ii) (which we have just shown), together with Remark \ref{Rmk_the_jumps_of_C} and Theorem \ref{caranonlinear}.  


\fproof
  
\begin{Remark} By the same arguments as those of the proof of statement (i) in the above Theorem \ref{thm_value_nu}, the following  general statement can be shown:
A strong $\mathcal{E}^f$-supermartingale is right-continuous if it is right-continuous along stopping times in $\mathcal{E}^f$-conditional expectation.
\end{Remark}

\section{Aclnowledgements} The authors are very grateful to Klébert Kentia for his helpful remarks.
The authors are also indebted  to Sigurd Assing for his helpful comments, and to Marek Rutkowski and Tianyang Nie for useful discussions.

\end{document}